\documentclass[oneside,english]{amsart}
\usepackage[a4paper,margin=2cm]{geometry}
\usepackage[T1]{fontenc}
\usepackage[latin9]{inputenc}
\usepackage{txfonts}

\usepackage{babel}
\usepackage{amsthm}
\usepackage{amstext}
\usepackage{graphicx}
\usepackage{esint}
\usepackage[unicode=true,pdfusetitle,
 bookmarks=true,bookmarksnumbered=false,bookmarksopen=false,
 breaklinks=false,pdfborder={0 0 1},backref=false,colorlinks=false]
 {hyperref}
\usepackage{breakurl}
\usepackage{dsfont} 
\usepackage{microtype} 
\usepackage{siunitx}
\usepackage{booktabs}

\makeatletter
  \theoremstyle{plain}
  \newtheorem{assumption}{\protect\assumptionname}
  \theoremstyle{definition}
  \newtheorem{defn}{\protect\definitionname}
\theoremstyle{plain}
\newtheorem{thm}{\protect\theoremname}
  \theoremstyle{remark}
  \newtheorem{rem}{\protect\remarkname}
 \theoremstyle{definition}
  \newtheorem{example}{\protect\examplename}
  \theoremstyle{plain}
  \newtheorem{lem}{\protect\lemmaname}
  \theoremstyle{plain}
  \newtheorem{cor}{\protect\corollaryname}
  \theoremstyle{plain}
  \newtheorem{prop}{\protect\propositionname}
  \theoremstyle{remark}
  \newtheorem{acknowledgement}{\protect\acknowledgementname}

\makeatother

  \providecommand{\acknowledgementname}{Acknowledgement}
  \providecommand{\assumptionname}{Assumption}
  \providecommand{\definitionname}{Definition}
  \providecommand{\examplename}{Example}
  \providecommand{\lemmaname}{Lemma}
  \providecommand{\propositionname}{Proposition}
  \providecommand{\remarkname}{Remark}
\providecommand{\corollaryname}{Corollary}
\providecommand{\theoremname}{Theorem}

\begin{document}

\title{root's barrier, viscosity solutions of obstacle problems and reflected
FBSDEs}

\author{Paul Gassiat }

\address{TU Berlin, Institut f\"ur Mathematik, Straße des 17. Juni 136, 10623 Berlin, Germany}

\email{gassiat@math.tu-berlin.de}

\author{Harald Oberhauser }

\address{Oxford-Man Institute, University of Oxford, Walton Well Rd, OX2 6ED
Oxford, United Kingdom}

\email{harald.oberhauser@oxford-man.ox.ac.uk}

\author{Gonçalo dos Reis}

\address{University of Edinburgh, School of Mathematics, Edinburgh, EH9 3JZ,
UK}

\email{G.dosReis@ed.ac.uk}

\begin{abstract}
We revisit work of Rost \cite{MR0445600}, Dupire \cite{Dupire2005} and Cox--Wang \cite{CoxWang2011} on connections between Root's solution of the Skorokhod embedding problem and obstacle problems. 
We develop an approach based on viscosity sub- and supersolutions and an accompanying comparison principle. 
This gives a complete characterization of (reversed) Root barriers and also leads to new proofs of the existence as well as the minimality of Root solutions by pure PDE methods. 
The approach is self-contained, constructive and also general enough to cover martingale diffusions with degenerate elliptic or time-dependent volatility; it also provides insights about the dynamics of general Skorokhod embeddings by identifying them as supersolutions of certain nonlinear PDEs.
\end{abstract}
\maketitle

\section{Introduction }

This article revisits the dynamics of the Skorokhod embedding problem
from a viscosity PDE perspective with an emphasis on Root's solution.
That is, under mild assumptions on the probability measures
$\mu,\nu$ on $\mathbb{R}$ and the volatility coefficient $\sigma$,
we are interested in finding a (non-randomized) stopping time $\tau$
such that
\begin{equation}
\tag{{SEP}}\left\{ \begin{array}{rcl}
dX_{t} & = & \sigma\left(t,X_{t}\right)dB_{t},\, X_{0}\sim\mu,\\
X_{\tau} & \sim & \nu\text{ and }X^{\tau}=\left(X_{t\wedge\tau}\right)_{t\geq0}\text{ is uniformly integrable.}
\end{array}\right.\label{eq:sep}
\end{equation}
For general background on (\ref{eq:sep}) and its applications we
refer to Hobson and Ob\l{}ój \cite{MR2068476,Hobson2011}. In the
Brownian case, $\sigma\equiv1$ and $\mu=\delta_{0}$, Root showed
in 1969 \cite{MR0238394} that the stopping time $\tau$ can be realized
as the first hitting of a closed time-space set, the so-called Root
barrier, 
\[
R\subset\left[0,\infty\right]\times\left[-\infty,\infty\right]
\]
by the time-space process $t\mapsto\left(t,X_{t}\right)$. Further important
developments are due to Rost: firstly, he showed that Root's solution
minimizes 
\begin{equation}
\mathbb{E}\left[\int_{t\wedge\tau}^{\tau}f\left(X_s\right)ds\right]\,\,\,\forall f\geq0,\forall t\geq0\label{eq:second moment}
\end{equation}
among all solutions of (\ref{eq:sep}), \cite{MR0445600}; secondly,
he gave a new existence proof of Root's barrier by using potential
theory that generalizes to time-homogenous Markov processes \cite{MR0445600,MR0273691};
thirdly, he showed that there exist another barrier $R^\text{rev}$ that
solves (\ref{eq:sep}) and minimizes the diffusion of $X$, \cite{MR0356248,MR0346920}.
Another important contribution concerning the uniqueness of the barrier
$R$ was made by Loynes \cite{MR0292170}.

Already for Brownian motion it was not known how to construct $R$
except for a handful of simple cases. A completely new perspective
that led to a revived interest in (\ref{eq:sep}) came from the management
of risk in mathematical finance. It was started by work of Hobson
\cite{hobson1998robust} that showed how model-independent bounds
of exotic options can be obtained by ``extremal solutions'' of (\ref{eq:sep}).
Motivated by this, Dupire \cite{Dupire2005} showed formally that
the barrier $R$ is naturally linked to a nonlinear PDE that allows
to solve for $R$. This was further developed by Cox--Wang \cite{CoxWang2011}
who use a variational formulation, as developed in the 1970's by Bensoussan--Lions
et.~al.~\cite{bensoussan1982applications}, to calculcate $R$ in
case its existence is guaranteed by these classic results of Root and Rost. 
More recently, Gassiat--Mijatovic--Oberhauser \cite{gassiat2013integral} studied
the barrier via integral equations, Cox--Peskir \cite{cox2012embedding}
studied reversed barriers, Beiglböck--Cox--Huesmann \cite{beiglboeck2013optimal}
develop an optimal transport perspective of (\ref{eq:sep}) and Kleptsyn--Kurtzmann \cite{kleptsyn2012counter} use Root's barrier to construct a counter-example to the Cantelli conjecture; there
are also many more developments beyond the context of Root type solutions,
see for example \cite{2013arXiv1306.3942A,AnkirchnerHeyneImkeller2008,ankirchner2011skorokhod,galichon2011stochastic,henry2012maximum,obloj2013iterated}
for interesting recent progress.

This article takes the PDE approach further by giving a self-contained approach to the embedding problem based on viscosity solutions. 
The parabolic comparison principle plays the key role.
It allows us to provide new proofs of firstly, the existence of a Root solution, and secondly, its minimizing property (\ref{eq:second moment}). 
In the Brownian, or homogenous diffusion case, this recovers the classic results of Root and Rost \cite{MR0238394,MR0445600} by constructive methods and provides insights about the general dynamics of (\ref{eq:sep}); 
in the time-inhomogenous case, already the existence and minimality results themselves are new to the best of our knowledge and would
be hard to obtain otherwise, since the classic approaches break down%
\footnote{The most general existence proof is due to Rost \cite{MR0445600}
and relies heavily on time-homogenity (and to certain degree transience)
of the underlying process, thereby excluding (\ref{eq:sep}) for time-dependent
$\sigma=\sigma\left(t,x\right)$.%
}; however in the current setup they become simple consequences of
a PDE existence and a comparison of sub- and supersolutions; see Theorem
\ref{thm:supersolution}, Theorem \ref{thm:root as pde} and Theorem
\ref{thm:root as minimizer}. 
Moreover, the PDE methods we introduce also cover Rost's reversed Root barriers, where already for the Brownian case, they might be an attractive, constructive(!) alternative to the classic ``filling scheme'' \cite{chaconThesis,MR0346920} proof that relies on deep results from potential theory.

\subsection*{Structure of the article.}

In Section \ref{sec:Root's-barrier-and} we introduce notation and
our assumptions on $\left(\sigma,\mu,\nu\right)$. We then give our
first main result, Theorem \ref{thm:supersolution}, which states
that\emph{ any} \emph{solution} of the Skorokhod embedding problem
(\ref{eq:sep}) is a viscosity supersolution of a certain obstacle
PDE. If one thinks in potential theoretic terms, this can be seen
as the PDE version of Rost's approach \cite{MR0445600} to (\ref{eq:sep})
via excessive functions of Markov processes.

In Section \ref{sec:Root's solution} we introduce an extension of
Loyne's notion of regular Root barriers which allows to deduce the
uniqueness of such barrier solutions. We then present our second main
result, Theorem \ref{thm:root as pde}, which shows a one-on-one correspondence
of regular Root barriers and viscosity solutions. This complete characterization
allows us to firstly, prove the existence of Root barriers via PDE
existence, Corollary \ref{cor:existence root}, and secondly, show
that the minimizing property (\ref{eq:second moment}) is a simple
consequence of a parabolic comparison result, Theorem \ref{sec:Application-I:-calculating}.
We also briefly revisit and leverage results about reflected FBSDEs
\cite{ElKarouiPengQuenez1997} which allows to use Monte-Carlo methods to calculate barriers and gives another interpretation as a stopping problem. 
We conclude this section by showing, how our approach also gives existence and minimalty proofs of Rost's reversed barriers.

In Section \ref{sec:Application-I:-calculating} we implement numerical
schemes to solve for the barrier and apply the Barles--Souganidis
method to get convergence (+rates of convergence) which might be useful
for practitioners in financial mathematics (bounds on options on variance).

\section{\label{sec:Root's-barrier-and}Skorokhod embeddings as supersolutions}

\subsection{Notation and Assumptions}

Denote with $\left(\Omega,\mathcal{F},\left(\mathcal{F}_{t}\right),\mathbb{P}\right)$
a filtered probability space that satisfies the usual conditions and
carries a standard Brownian motion $B$ and a real-valued random variable
$X_{0}\sim\mu$ that is independent of $B$. We work under the following
assumption on $\left(\sigma,\mu,\nu\right)$.
\begin{assumption}
\label{as:main assumption}$\mu$ and $\nu$ are measures on $\left(\mathbb{R},\mathcal{B}\left(\mathbb{R}\right)\right)$
that have a first moment and are in convex order, i.e.~
\[
u_{\nu}\left(x\right):=-\int_{\mathbb{R}}\left|x-y\right|\nu\left(dy\right)\leq-\int_{\mathbb{R}}\left|x-y\right|\mu\left(dy\right)=:u_{\mu}\left(x\right)\,\,\forall x\in\mathbb{R}.
\]
Let $\sigma\in C\left(\left[0,\infty\right)\times\mathbb{R},\mathbb{R}\right)$
be Lipschitz in space and of linear growth, both uniformly in time,
that is 
\begin{equation}
\sigma_{Lip}:=\sup_{t\in\left[0,\infty\right)}\sup_{x\neq y}\frac{\left|\sigma\left(t,x\right)-\sigma\left(t,y\right)\right|}{\left|x-y\right|}<\infty\text{ \,\ and\,\ }\sigma_{LG}:=\sup_{t\in\left[0,\infty\right)}\sup_{x\in\mathbb{R}}\frac{\left|\sigma\left(t,x\right)\right|}{1+\left|x\right|}<\infty.\label{eq:sigma-conditions-1}
\end{equation}
Further assume local ellipticity in the sense that for each compact $K \subset \left\{ x:u_{\mu}\left(x\right)\neq u_{\nu}\left(x\right)\right\}$, there exists some $c_K>0$ s.t. $$\sigma(t,x) \geq c_K >0, \;\;\;\; \forall t \geq 0, x \in K.$$
%\begin{enumerate}
%\item the unique strong solution of the SDE driven by $\sigma$, 
%\begin{eqnarray}
%X_{.} & = & X_{0}+\int_{0}^{.}\sigma\left(t,X_{t}\right)dB_{t}\,\,\,\mathbb{P}-a.s\label{eq:ito-diff}\\
%X_{0} & \sim & \mu\nonumber 
%\end{eqnarray}
%has no intervals of constancy%
%\footnote{This is equivalent to $t\mapsto\left[X\right]_{t}$ being strictly
%increasing (but possibly converging to a finite limit, e.g.~in the
%case of Geometric Brownian motion). %
%},
%\item (@PAUL: MAYBE STH. LIKE THIS)$\forall x\in\left\{ x:u_{\mu}\left(x\right)\neq u_{\nu}\left(x\right)\right\} $
%$\exists t\in\left[0,\infty\right]$ such that $-\mathbb{E}\left[L_{t}^{x}\right]\leq u_{\nu}\left(x\right)$
%where $\left(L_{t}^{x}\right)$ denotes the local time of $X$.
%\end{enumerate}
\end{assumption}
The need for above assumptions is intuitively clear: convex order
of $\mu,\nu$ is necessary by classic work of Chacon and Kellerer
about the marginals of martingales \cite{kellerer1972markov,chacon1976one};
however we only assume first moments which is already for the Brownian
case, $\sigma\equiv 1$, weaker than Root's assumption \cite{MR0238394}. 
Linear growth and Lipschitz property of $\sigma$ are natural since we describe the evolution of the law of the strong solution $X^{\tau}$ by PDEs; some nondegeneracy of the diffusion is clearly required to be able to transport the mass $\mu$ to $\nu$.
%the non-constancy assumption follows \cite{Obloj2012} and is there
%to exclude mostly trivial cases, e.g.~by switching $\sigma$ on and
%off over certain time intervals. The local time assumption ensures
%that $X$ reaches the points where mass from $\mu$ to $\nu$ needs
%to be transferred. Note that it is immediately fulfilled when the
%volatility is bounded below from zero (via Dambis/Dubins--Schwarz)
(Note that what we call ``local ellipticity'' covers degenerate elliptic diffusions, e.g.~for geometric Brownian motion, $\sigma(x)=x$ our assumption is fulfilled if and only if the support of $\nu$ is contained in the positive halfline which
is in this case the sharp condition).%
%\footnote{We suspect that it is not sharp in the general case, e.g.~it is conceivable
%that a random $\Delta>0$ might be sufficient. Note that by the Ray--Knight
%Theorem \cite[Corollary 22.18]{kallenberg2002foundations} reduces
%to $\mathbb{P}\left[\inf_{s\in\left[t,t+r\right]}X_{s}<x<\sup_{s\in\left[t,t+r\right]}X_{s}\right]>0$
%which is sometimes easy to check; other sufficient condition when
%$\sigma$ is not bounded below can for example be found in \cite{karatzas2013distribution}. %
%}.
\begin{defn}
Let $\left(\sigma,\mu,\nu\right)$ fulfill Assumption \ref{as:main assumption}.
We denote with $\operatorname{\text{SEP}_{\sigma,\mu,\nu}}$ the set
of \emph{$\left(\mathcal{F}_{t}\right)$}-stopping times $\tau$ that
solve the Skorokhod embedding problem 
\[
\left\{ \begin{array}{rcl}
dX_{t} & = & \sigma\left(t,X_{t}\right)dB_{t},\, X_{0}\sim\mu,\\
X_{\tau} & \sim & \nu\text{ and }X^{\tau}=\left(X_{t\wedge\tau}\right)_{t\geq0}\text{ is uniformly integrable.}
\end{array}\right.
\]

\end{defn}

\subsection{Recall on viscosity theory}

We briefly recall standard concepts from viscosity theory; for more
details see \cite{CrandallIshiiLions1992,MR2179357FS}.
\begin{defn}
\label{def:Jets}Let $\mathcal{O}$ be a locally compact subset of
$\mathbb{R}$ and denote $\mathcal{O}_{T}=\left(0,T\right)\times\mathcal{O}$
and $\overline{\mathcal{O}}_{T}=\left[0,T\right]\times\mathcal{O}$
for a given $T\in\left(0,\infty\right]$. Consider a function $u:\mathcal{O}_{T}\rightarrow\mathbb{R}$
and define for $\left(s,z\right)\in\mathcal{O}_{T}$ the parabolic
superjet $\mathcal{P}_{\mathcal{O}}^{2,+}u\left(s,z\right)$ as the
set of triples $\left(a,p,m\right)\in\mathbb{R}\times\mathbb{R}\times\mathbb{R}$
which fulfill
\begin{eqnarray}
u\left(t,x\right) & \leq & u\left(s,z\right)+a\left(t-s\right)+\left\langle p,x-z\right\rangle \nonumber \\
 &  & +\frac{1}{2}\left\langle m\left(x-z\right),x-z\right\rangle +o\left(\left|t-s\right|+\left|x-z\right|^{2}\right)\text{ as }\mathcal{O}_{T}\ni\left(t,x\right)\rightarrow\left(s,z\right)\label{eq:superjetcondition}
\end{eqnarray}
Similarly we define the parabolic subjet $\mathcal{P}_{\mathcal{O}}^{2,-}u\left(s,z\right)$
such that $\mathcal{P}_{\mathcal{O}}^{2,-}u=-\mathcal{P}_{\mathcal{O}}^{2,+}\left(-u\right)$. 
\begin{defn}
~

\label{def-super-sub-viscsol}A function $F:\mathcal{O}_{T}\times\mathbb{R}\times\mathbb{R}\times\mathbb{R}\times\mathbb{R}\rightarrow\mathbb{R}$
is \emph{proper} if $\forall\left(t,x,a,p\right)\in\mathcal{O}_{T}\times\mathbb{R}\times\mathbb{R}$
\[
F\left(t,x,r,a,p,m\right)\leq F\left(t,x,s,a,p,m^{\prime}\right)\,\,\forall m\geq m^{\prime},\, s\geq r.
\]
Denote the real-valued, upper semicontinuous functions on $\overline{\mathcal{O}}_{T}$
with\emph{ $USC\left(\overline{\mathcal{O}}_{T}\right)$} and the
lower semicontinuous functions with $LSC\left(\overline{\mathcal{O}}_{T}\right)$.
A \textit{subsolution} of the (\emph{forward} problem)
\begin{equation}
\left\{ \begin{array}{rcl}
F\left(t,x,u,\partial_{t}u,Du,D^{2}u\right) & = & 0\\
u\left(0,.\right) & = & u_{0}\left(.\right)
\end{array}\right.\label{eq:nonlinear pde}
\end{equation}
is a function $u\in USC\left(\overline{\mathcal{O}}_{T}\right)$ such
that 
\begin{eqnarray*}
F\left(t,x,u,a,p,m\right) & \leq & 0\text{ for }\left(t,x\right)\in\mathcal{O}_{T}\text{ and }\left(a,p,m\right)\in\mathcal{P}_{\mathcal{O}}^{2,+}u\left(t,x\right)\\
u\left(0,.\right) & \leq & u_{0}\left(.\right)\text{ on }\mathcal{O}
\end{eqnarray*}
The definition of a \textit{supersolution} follows by replacing $USC\left(\overline{\mathcal{O}}_{T}\right)$
by $LSC\left(\overline{\mathcal{O}}_{T}\right)$, $\mathcal{P}_{\mathcal{O}}^{2,+}$
by $\mathcal{P}_{\mathcal{O}}^{2,-}$ and $\leq$ by $\geq.$ If $u$
is a supersolution of (\ref{eq:nonlinear pde}) then we also say that
$F\left(t,x,u,\partial_{t}u,Du,D^{2}u\right)\geq0,\, u\left(0,x\right)\geq u_{0}\left(x\right)$
holds in viscosity sense (similar for subsolutions). Similarly we
call a function $v$ a viscosity (\emph{sub-},\emph{super-}) \emph{solution}
of the \emph{backward} problem 
\begin{equation}
\left\{ \begin{array}{rcl}
G\left(t,x,v,\partial_{t}v,Dv,D^{2}v\right) & = & 0\\
v\left(T,.\right) & = & v_{T}\left(.\right)
\end{array}\right.\label{eq:nonlinear pde-1}
\end{equation}
if
\begin{eqnarray*}
G\left(t,x,v,a,p,m\right) & \leq & 0\text{ for }\left(t,x\right)\in\mathcal{O}_{T}\text{ and }\left(a,p,m\right)\in\mathcal{P}_{\mathcal{O}}^{2,+}v\left(t,x\right)\\
v\left(0,.\right) & \leq & v_{T}\left(.\right)\text{ on }\mathcal{O}.
\end{eqnarray*}

\end{defn}
\end{defn}

\subsection{Skorokhod embeddings as PDE supersolutions}

Chacon \cite{MR0501374} showed that for $\tau\in\operatorname{SEP}_{\sigma,\mu,\nu}$,
the potential function $u^{\tau}\left(t,x\right)\equiv-\mathbb{E}\left[\left|X_{t\wedge\tau}-x\right|\right]$
is a powerful tool to study the evolution of the law of the stopped
(local) martingale $X^{\tau}=\left(X_{\tau\wedge t}\right)$. Theorem
\ref{thm:supersolution} captures the following intuition: $u^{\tau}$
is a concave function of $X^{\tau}$, hence we expect it to be a supersolution
(in some sense) of a Fokker--Planck equation. However, $u^{\tau}$
cannot be smooth for generic Skorokhod solutions due to kinks from
stopping at $u_{\nu}$. Further, since $u^{\tau}$ is the potential
of the occupation measure of $X^{\tau}$ and $\tau\in\operatorname{SEP}_{\sigma,\mu,\nu}$,
it follows that $u^{\tau}\left(t,.\right)$ is bounded below by the
potential of the measure $\nu$ and will converge to it as $t\rightarrow\infty$.
We now make this rigorous under the generality of Assumption \ref{as:main assumption}.
\begin{thm}
\label{thm:supersolution}Let $\left(\sigma,\mu,\nu\right)$ fulfill
Assumption \ref{as:main assumption}. Then for every $\tau\in$$\operatorname{SEP}_{\sigma,\mu,\nu}$
the function 
\[
u^{\tau}\left(t,x\right):=-\mathbb{E}\left[\left|X_{\tau\wedge t}-x\right|\right]
\]
is a viscosity supersolution of

\textup{
\begin{equation}
\left\{ \begin{array}{rcl}
\inf\left(u-u_{\nu},\partial_{t}u-\frac{\sigma^{2}}{2}\Delta u\right) & = & 0\text{ on }\left(0,\infty\right)\times\mathbb{R},\\
u\left(0,.\right) & = & u_{\mu}\left(.\right)
\end{array}\right.\label{eq:obs super}
\end{equation}
}and 
\[
\lim_{t\rightarrow\infty}u^{\tau}\left(t,x\right)=u_{\nu}\left(x\right)\,\,\,\forall x\in\mathbb{R}.
\]
\end{thm}
\begin{proof}
We have to show that 
\[
\left\{ \begin{array}{rcl}
u^{\tau}-u_{\nu} & \geq & 0\text{ on }\left(0,\infty\right)\times\mathbb{R},\\
\left(\partial_{t}-\frac{\sigma^{2}}{2}\Delta\right)u^{\tau} & \geq & 0\text{ on }\left(0,\infty\right)\times\mathbb{R},\\
\lim_{t\rightarrow\infty}u^{\tau} & = & u_{\nu}.
\end{array}\right.
\]
The first inequality follows immediately via conditional Jensen 
\begin{equation}
u^{\tau}\left(t,x\right)=-\mathbb{E}\left[\left|X_{t}^{\tau}-x\right|\right]\geq\mathbb{E}\left[\mathbb{E}\left[-\left|X_{\tau}-x\right|\lvert\mathcal{F}_{t\wedge\tau}\right]\right]=-\mathbb{E}\left[\left|X_{\tau}-x\right|\right]=u_{\nu}\left(x\right).\label{eq:lower bound}
\end{equation}
and the last one is immediate from properties of potential functions,
see \cite{MR0501374,MR2068476}. For the second inequality we approximate
$u^{\tau}$ by a sequence of regularizations $\left(u^{n}\right)_{n}$.
We show that each $u^{n}$ is a supersolution of a ''perturbed version''
of (\ref{eq:obs super}) and we conclude by sending $n\rightarrow\infty$
and using the stability of viscosity solutions.

\textbf{Step 1}\textbf{\emph{.}} Convergence of\textbf{\emph{ $u^{n}\left(t,x\right):=\mathbb{E}\left[\psi^{n}\left(X_{t}^{\tau}-x\right)\right]$}}
as $n\rightarrow\infty.$\\
Define the sequence $\left(\psi_{n}\right)\subset C^{2}\left(\mathbb{R},\mathbb{R}\right)$
as 
\[
\psi_{n}\left(x\right)=\int_{0}^{x}\int_{0}^{y}n\phi\left(nz\right)dzdy\forall x\in\mathbb{R}
\]
where $\phi$ is the usual Gaussian scaled to the unit disc 
\[
\phi\left(x\right)=\begin{cases}
\exp\left(-\frac{1}{1-x^{2}}\right) & \text{for }\left|x\right|<1,\\
0 & \text{otherwise}.
\end{cases}
\]
Especially note that $\psi_{n}\left(.\right)\rightarrow\left|.\right|$
uniformly, $\Delta\psi_{n}\left(.\right)$ is continuous, $0\leq\Delta\psi_{n}\leq O\left(n\right)$
and $\operatorname{supp}\left(\Delta\psi_{n}\right)\subset\left[-\frac{1}{n},\frac{1}{n}\right]$
(we could replace $\psi^{n}$ by any other sequence with this properties).
Further introduce 
\begin{eqnarray*}
u^{n}\left(t,x\right) & = & -\mathbb{E}\left[\psi_{n}\left(X_{t}^{\tau}-x\right)\right],\\
u_{\nu}^{n}\left(x\right) & = & -\int_{\mathbb{R}}\psi^{n}\left(y-x\right)\nu\left(dy\right),\\
u_{\mu}^{n}\left(x\right) & = & -\int_{\mathbb{R}}\psi^{n}\left(y-x\right)\mu\left(dy\right).
\end{eqnarray*}
Since $\psi_{n}\left(.\right)\rightarrow\left|.\right|$ uniformly
we have $\mathbb{P}$-a.s. 
\[
\left|\psi^{n}\left(X_{.}^{\tau}-.\right)-\left|X_{.}^{\tau}-.\right|\right|_{\infty;\left[0,\infty\right)\times\mathbb{R}}=\sup_{\left(t,x\right)\in\left[0,\infty\right)\times\mathbb{R}}\left|\psi^{n}\left(X_{t}^{\tau}-x\right)-\left|X_{t}^{\tau}-x\right|\right|\rightarrow_{n\rightarrow\infty}0
\]
hence we get uniform convergence of $u^{n},u_{\mu}^{n}$ and $u_{\nu}^{n}$,
i.e.
\[
\left|u^{n}-u\right|_{\infty;[0,\infty)\times\mathbb{R}}\rightarrow_{n}0,\qquad\left|u_{\nu}^{n}-u_{\nu}\right|_{\infty,\mathbb{R}}\rightarrow_{n}0,\qquad\text{{and}}\qquad\left|u_{\mu}^{n}-u_{\mu}\right|_{\infty,\mathbb{R}}\rightarrow_{n}0.
\]
Further, by the definition of $u$ and $u^{n}$ we see that $\forall x\in\mathbb{R}$,
$\forall n\in\mathbb{N}$ 
\begin{equation}
\lim_{t\rightarrow\infty}u\left(t,x\right)=u_{\nu}\left(x\right)\text{ and }\lim_{t\rightarrow\infty}u^{n}\left(t,x\right)=u_{\nu}^{n}\left(x\right).\label{eq:lim u to U}
\end{equation}

\textbf{Step 2. }$\left(\partial_{t}-\frac{\sigma^{2}}{2}\Delta\right)u\geq0$
on $\left(0,\infty\right)\times\mathbb{R}$. \\
We now fix $x\in\mathbb{R}$ and apply the Itô formula to $-\psi^{n}\left(\cdot-x\right)$
and the local martingale $X^{\tau}$ which, after taking expectations
and using Fubini, leads to the expression 
\[
u^{n}\left(t,x\right)=u_{\mu}^{n}\left(x\right)-\int_{0}^{t}\mathbb{E}\left[\frac{\sigma^{2}\left(r,X_{r}\right)}{2}\Delta\psi^{n}\left(X_{r}-x\right)1_{r<\tau}\right]dr.
\]
It follows that $u^{n}\left(t,x\right)$ has a right- and left- derivative
$\forall\left(t,x\right)\in\left(0,\infty\right)\times\mathbb{R}$;
to see this take 
\begin{eqnarray*}
\partial_{t+}u^{n}\left(t,x\right) & = & \lim_{\epsilon\downarrow0}\epsilon^{-1}\left(u\left(t+\epsilon,x\right)-u\left(t,x\right)\right)\\
 & = & -\mathbb{E}\left[\lim_{\epsilon\downarrow0}\epsilon^{-1}\int_{t}^{t+\epsilon}\frac{\sigma^{2}\left(r,X_{r}\right)}{2}\Delta\psi^{n}\left(X_{r}-x\right)1_{r<\tau}dr\right]\\
 & = & -\mathbb{E}\left[\frac{\sigma^{2}\left(t,X_{t}\right)}{2}\Delta\psi^{n}\left(X_{t}-x\right)1_{t<\tau}\right]
\end{eqnarray*}
and similarly it follows that the left derivative $\partial_{t-}u^{n}$
is given as
\begin{equation}
\partial_{t-}u^{n}\left(t,x\right)=-\mathbb{E}\left[\frac{\sigma^{2}\left(t,X_{t}\right)}{2}\Delta\psi^{n}\left(X_{t}-x\right)1_{t\leq\tau}\right].\label{eq:u_t-}
\end{equation}
Note that for every $t\in\left(0,\infty\right)$ $\partial_{t-}u^{n}\left(t,.\right),\partial_{t+}u^{n}\left(t,.\right)\in C^{\infty}\left(\mathbb{R},\mathbb{R}\right)$;
further, since $\frac{\sigma^{2}\left(t,X_{t}\right)}{2}\Delta\psi^{n}\left(X_{t}-x\right)$
is non-negative we conclude 
\[
\partial_{t-}u^{n}\left(t,x\right)\leq\partial_{t+}u^{n}\left(t,x\right)\leq0.
\]
From the definition of $u^{n}$ it follows that we can exchange differentiation
in space and expectation to arrive at 
\begin{equation}
\frac{\sigma^{2}\left(t,x\right)}{2}\Delta u^{n}\left(t,x\right)=-\frac{\sigma^{2}\left(t,x\right)}{2}\mathbb{E}\left[\Delta\psi^{n}\left(X_{t}^{\tau}-x\right)\right]\leq0\text{ on }\left(0,\infty\right)\times\mathbb{R},\label{eq:u_xx negative}
\end{equation}
which is a continuous function in $\left(t,x\right)$. Lemma \ref{lem:jets for u-1}
shows that $\forall\left(a,p,m\right)\in\mathcal{P}_{\mathcal{O}}^{2,-}u^{n}\left(t,x\right)$
\[
a\geq\partial_{t-}u^{n}\left(t,x\right)\text{ and }m\leq\Delta u^{n}\left(t,x\right).
\]
Hence, by (\ref{eq:u_xx negative}) and (\ref{eq:u_t-}) 
\begin{eqnarray*}
a-m & \geq & \partial_{t-}u^{n}\left(t,x\right)-\frac{\sigma^{2}\left(t,x\right)}{2}\Delta u^{n}\left(t,x\right)\\
 & = & \frac{1}{2}\mathbb{E}\left[\sigma^{2}\left(t,x\right)\Delta\psi^{n}\left(X_{t}^{\tau}-x\right)-\sigma^{2}\left(t,X_{t}\right)\Delta\psi^{n}\left(X_{t}-x\right)1_{t\leq\tau}\right].
\end{eqnarray*}
Splitting the term inside the expectation gives 
\begin{eqnarray*}
\partial_{t-}u^{\epsilon}\left(t,x\right)-m & \geq & \frac{1}{2}\mathbb{E}\left[\sigma^{2}\left(t,x\right)\Delta\psi^{n}\left(X_{t}^{\tau}-x\right)-\sigma^{2}\left(t,x\right)\Delta\psi^{n}\left(X_{t}-x\right)1_{t\leq\tau}\right]\\
 &  & +\frac{1}{2}\mathbb{E}\left[\sigma^{2}\left(t,x\right)\Delta\psi^{n}\left(X_{t}-x\right)1_{t\leq\tau}-\sigma^{2}\left(t,X_{t}\right)\Delta\psi^{n}\left(X_{t}-x\right)1_{t\leq\tau}\right]\\
 & =: & \frac{1}{2}I_{n}\left(t,x\right)+\frac{1}{2}II_{n}\left(t,x\right).
\end{eqnarray*}
We conclude that $u^{n}$ is a supersolution of $\left(\partial_{t}-\frac{\sigma^{2}}{2}\Delta\right)u-\frac{1}{2}\left(I_{n}+II_{n}\right)=0$
on $\left(0,\infty\right)\times\mathbb{R}$. Further, 
\begin{eqnarray*}
I_{n}\left(t,x\right) & = & \mathbb{E}\left[\sigma^{2}\left(t,x\right)\Delta\psi^{n}\left(X_{t}^{\tau}-x\right)-\sigma^{2}\left(t,x\right)\Delta\psi^{n}\left(X_{t}-x\right)1_{t\le\tau}\right]\\
 & = & \sigma^{2}\left(t,x\right)\mathbb{E}\left[\Delta\psi^{n}\left(X_{\tau}-x\right)1_{t>\tau}\right]\geq0
\end{eqnarray*}
hence $u^{n}$ is also a viscosity supersolution of
\[
\left\{ \begin{array}{rcl}
\partial_{t}u-\frac{\sigma^{2}}{2}\Delta u-\frac{1}{2}II_{n} & = & 0\text{ on }\left(0,\infty\right)\times\mathbb{R}\\
u\left(0,.\right) & = & u_{\mu}^{n}\left(.\right).
\end{array}\right.
\]
Using the Lipschitz property of $\sigma$, $\operatorname{supp}\left(\Delta\psi^{n}\right)=\left[-n^{-1},n^{-1}\right]$
and that $\left|\Delta\psi^{n}\right|_{\infty}\leq c\, n$ we estimate
\begin{eqnarray*}
\left|II_{n}\left(t,x\right)\right| & \leq & \mathbb{E}\left[\left|\sigma^{2}\left(t,x\right)-\sigma^{2}\left(t,X_{t}\right)\right|\Delta\psi^{n}\left(X_{t}-x\right)1_{t\leq\tau}\right]\\
 & \leq & \sigma_{Lip}\frac{2}{n}2\sigma_{LG}\left|1+x+\frac{2}{n}\right|\mathbb{E}\left[\Delta\psi^{n}\left(X_{t}-x\right)1_{t\leq\tau}\right]\\
 & \leq & 4\sigma_{Lip}\sigma_{LG}\left|x+\frac{2}{n}\right|\frac{1}{n}\mathbb{E}\left[cn1_{\left|X_{t}-x\right|\leq n^{-1}}\right]\\
 & = & 4c\sigma_{Lip}\sigma_{LG}\left|x+\frac{2}{n}\right|\mathbb{P}\left(\left|X_{t}-x\right|\leq n^{-1}\right)
\end{eqnarray*}
($\sigma_{Lip}$ and $\sigma_{LG}$ as defined in (\ref{eq:sigma-conditions-1});
for the second inequality we use the trivial estimate 
\[
\left|\sigma^{2}\left(t,x\right)-\sigma^{2}\left(t,X_{t}\right)\right|\leq\left|\sigma\left(t,x\right)-\sigma\left(t,X_{t}\right)\right|\left|\sigma\left(t,x\right)+\sigma\left(t,X_{t}\right)\right|
\]
combined with $\left|x-X_{t}\right|\leq\frac{2}{n}$, Lipschitzness
and linear growth of $\sigma$). Now for every compact $K\subset\left[0,\infty\right)\times\mathbb{R}$
we have 
\[
\lim_{n\rightarrow\infty}\sup_{\left(t,x\right)\in K}\mathbb{P}\left[\left|X_{t}-x\right|\leq n^{-1}\right]=0
\]
since our Assumption \ref{as:main assumption} guarantees (via \cite[Theorem 2.3.1]{Nu06})
that the process $X$ has a density $f\left(t,.\right)$ for all $t>0$
with respect to Lebesgue measure and
\[
\mathbb{P}\left[\left|X_{t}-x\right|\leq n^{-1}\right]=\int_{x-\frac{1}{n}}^{x+\frac{1}{n}}\left|y-x\right|f\left(t,y\right)dy\leq\frac{1}{n}\int_{\mathbb{R}}f\left(t,y\right)dy\rightarrow0\text{ as }n\rightarrow\infty
\]
uniformly in $\left(t,x\right)$, therefore $II_{n}\rightarrow0$
locally uniformly on $\left(0,\infty\right)\times\mathbb{R}$. By
step 1, $u^{n}\rightarrow u$ and $u^{n}\left(0,.\right)\rightarrow u_{\mu}\left(.\right)$
locally uniformly as $n\rightarrow\infty$. The usual stability of
viscosity solutions, see \cite{MR1118699UserGuide}, implies that
$u$ is a viscosity supersolution of
\begin{equation}
\left\{ \begin{array}{rcl}
\left(\partial_{t}-\frac{\sigma^{2}}{2}\Delta\right)u & = & 0\text{ on }\left(0,\infty\right)\times\mathbb{R}\\
u\left(0,.\right) & = & u_{\mu}\left(.\right)
\end{array}\right.\label{eq:fokkerplanck}
\end{equation}
which proves the desired inequality.\end{proof}
\begin{rem}
$\tau\in\operatorname{SEP}_{\sigma,\mu,\nu}$ can be a complicated
(even randomized) functional of $X$. While for some solutions, it
is known that $\tau$ is connected to an optimal stopping problem
(see Section (\ref{sub:bsde})), in general one can not expect every
$\tau\in\operatorname{SEP}_{\sigma,\mu,\nu}$ to have a minimizing/extremal
property (see Section \ref{sub:Root as minimizer}).
\end{rem}

\section{\label{sec:Root's solution}Root's solution}

In principle, a solution of the Skorokhod embedding problem, $\tau\in\operatorname{SEP}_{\sigma,\mu,\nu}$,
could be a complicated functional of the trajectories of $X$. Root
\cite{MR0238394} showed that the arguably simplest class of stopping
times, namely the hitting times of ``nice'' subsets in time-space,
so-called Root barriers, is big enough to solve Skorokhod's embedding
problem for Brownian motion. We now give a complete characterization
of such barriers as free boundaries of PDEs.

\subsection{Root barriers}
\begin{defn}
\label{def: root barrier}A closed subset $R$ of $\left[0,\infty\right]\times\left[-\infty,\infty\right]$
is a \emph{Root barrier} $R$ if
\begin{enumerate}
\item $\left(t,x\right)\in R$ implies $\left(t+r,x\right)\in R$ $\forall r\geq0$,
\item $\left(+\infty,x\right)\in R$ $\forall x\in\left[-\infty,\infty\right]$,
\item $\left(t,\pm\infty\right)\in R$ $\forall t\in\left[0,+\infty\right]$.
\end{enumerate}
We denote by $\mathcal{R}$ the set of all Root barriers $R$. Given
$R\in\mathcal{R}$, its \emph{barrier function} $f_{R}:\left[-\infty,\infty\right]\rightarrow\left[0,\infty\right]$
is defined as
\[
f_{R}\left(x\right):=\inf\left\{ t\geq0:\left(t,x\right)\in R\right\} ,\,\,\, x\in\left[-\infty,\infty\right].
\]

\end{defn}
Barrier functions have several nice properties such as being lower
semi-continuous and that $\left(f_{R}\left(x\right),x\right)\in R$
for any $x\in\mathbb{R}$, see \cite[Proposition 3]{MR0292170}.

\subsection{Uniqueness of regular Root barriers}

Different Root barriers can solve the same Skorokhod embedding problem%
\footnote{Let $\left(\sigma,\mu,\nu\right)=\left(1,\delta_{0},\frac{1}{2}\left(\delta_{-1}+\delta_{1}\right)\right)$,
then $R=\left[0,\infty\right]\times\left[1,\infty\right]\bigcup\left[0,\infty\right]\times\left[-\infty,-1\right]$
and any other Root barrier $R^{\prime}$ with a barrier function that
coincides with $f_{R}$ on $\left[-1,1\right]$ solves $\operatorname{SEP}_{\sigma,\mu,\nu}$.%
}. This problem of non-uniqueness was resolved in the Brownian case,
$\sigma\equiv1,\mu=\delta_{0}$, by Loynes \cite[p215]{MR0292170}
in 1970 by introducing the notion of regular Root barriers. 
\begin{defn}
\label{def:loynes-regular-barrier}$R\in\mathcal{R}$ resp.~its barrier
function $f_{R}$ is\emph{ Loynes--regular} if $f_{R}$ vanishes outside
the interval $\left[x_{-}^{R},x_{+}^{R}\right]$, where $x_{+}^{R}$
and $x_{-}^{R}$ are the first positive resp.~first negative zeros
of $f_{R}$. Given $Q,R\in\mathcal{R}$ we say that $Q,R$ are \emph{Loynes-equivalent
}if $f_{Q}=f_{R}$ on $\left[x_{-}^{Q},x_{+}^{Q}\right]$ and%
\footnote{If $Q,R$ are Loynes-equivalent then $x_{+}^{R}=x_{+}^{Q}$ and $x_{-}^{R}=x_{-}^{Q}$.%
} $\left[x_{-}^{R},x_{+}^{R}\right]$.
\end{defn}
Loynes showed that if a Root barrier solves the embedding problem
then there also exists a unique Loynes-regular barrier that solves
(\ref{eq:sep}). However, Loynes' notion of regularity is tailor-made
to the case of Dirac starting measures.
\begin{example}
\label{example-loynesinsuficient}Let $\mu=\frac{1}{2}\left(\delta_{2}+\delta_{-2}\right)$
and $\nu=\frac{1}{4}\left(\delta_{3}+\delta_{1}+\delta_{-1}+\delta_{-3}\right)$.
By symmetry properties of Brownian motion, for $a=b=0$ the barrier
\[
R_{a,b}=\left[0,\infty\right]\times\left[3,\infty\right]\cup\left[a,\infty\right]\times\left\{ 1\right\} \cup\left[b,\infty\right]\times\left\{ -1\right\} \cup\left[0,\infty\right]\times\left[-\infty,-3\right]\cup\left\{ +\infty\right\} \times\left[-\infty,+\infty\right]
\]
solves $\operatorname{SEP}_{1,\mu,\nu}$, as does 
\[
R=\left[0,\infty\right]\times\left[3,\infty\right]\cup\left[0,\infty\right]\times\left[1,-1\right]\cup\left[0,\infty\right]\times\left[-\infty,-3\right]\cup\left\{ +\infty\right\} \times\left[-\infty,+\infty\right].
\]
However, neither is Loynes-regular and there cannot exist a Loynes-regular
barrier%
\footnote{If $R_{a,b}$ solves $\operatorname{SEP}_{1,\mu,\nu}$ then $a,b>0$
otherwise it would not be Loynes regular; now note that $R_{0,0}$
solves $\operatorname{SEP}_{1,\mu,\nu}$, hence every other $R_{a,b}$
that puts under $\nu$ more mass on $3$ than the required $\frac{1}{4}$
since the geometry of $R_{a,b}$ implies that only more trajectories
can hit the line $\left[0,\infty\right]\times\left\{ 3\right\} $
than in the case $a=b=0$; further, every solution must coincide with
$R_{0,0}$ on $\left[0,\infty\right]\times\left[1,3\right]\cup\left[0,\infty\right]\times\left[-3,-1\right]$.%
}.
\end{example}
Motivated by the above we introduce the notion of $\left(\mu,\nu\right)$-\emph{regular}
barriers.
\begin{defn}
\label{def-nu-mu-regular}Define
\[
N^{\mu,\nu}:=\left\{ x\in\mathbb{R}:u_{\mu}\left(x\right)=u_{\nu}\left(x\right)\right\} \cup\left\{ \pm\infty\right\} \text{ and }\mathcal{N}^{\mu,\nu}:=\left[0,+\infty\right]\times N^{\mu,\nu}.
\]
We call a Root barrier $R$ $\left(\mu,\nu\right)$\emph{-regular}
if $R=R\cup\mathcal{N}^{\mu,\nu}$ {[}or equivalently if $f_{R}\left(x\right)=0$
$\forall x\in N^{\mu,\nu}${]} and denote with $\mathcal{R}_{\mu,\nu}$
the subset of Root barriers $\mathcal{R}$ that are $\left(\mu,\nu\right)$-regular.
Further, two Root barriers $R,Q$ are said to be $\left(\mu,\nu\right)$\emph{-equivalent}
if%
\footnote{We denote with $\overline{A}$ the closure and with $A^{o}$ the interior
of a given set $A$. %
} $R\setminus\left(\left[0,\infty\right]\times\left(N^{\mu,\nu}\right)^{o}\right)=Q\setminus\left(\left[0,\infty\right]\times\left(N^{\mu,\nu}\right)^{o}\right)$
{[}or equivalently if $f_{R}\left(x\right)=f_{Q}\left(x\right)$ $\forall x\in\overline{\left(N^{\mu,\nu}\right){}^{c}}${]}.
\end{defn}
We first show that in the case of Brownian motion started at a Dirac
in $0$, the above notion of regularity coincides with Loynes regularity.
We then show that for every Root barrier that solves $\operatorname{SEP}_{\sigma,\mu,\nu}$
there exist a unique $\left(\mu,\nu\right)$-regular barrier that
solves the same embedding.
\begin{lem}
\label{lem:equiv-defs-for-barrier-reg}Let $\sigma\equiv1$ and $\nu$
fulfill $\int x\nu\left(dx\right)=0$ and $R\in\mathcal{R}$. Then
$R$ is Loynes-regular iff $R$ is $\left(\delta_{0},\nu\right)$-regular.\end{lem}
\begin{proof}
If $R$ is Lyones regular then one has that $\nu\left(\left[x_{-}^{R},x_{+}^{R}\right]\right)=1$
for $x_{-}^{R},x_{+}^{R}$ from Definition \ref{def:loynes-regular-barrier}.
This and the continuity of the potential functions mean that $u_{\delta_{0}}\left(x\right)=u_{\nu}\left(x\right)$
for $x\notin\left(x_{-}^{R},x_{+}^{R}\right)$ and hence $R$ is $\left(\delta_{0},\nu\right)$-regular.
For the inverse direction, just remark that by definition of $\left(\delta_{0},\nu\right)$-regularity
and the convex order relation one has $N^{\delta_{0},\nu}\cap\mathbb{R}=\mathbb{R}\setminus\left(a,b\right)$
for some $a<0<b$, in other words $f_{R}\left(x\right)=0$ for any
$x\notin\left(a,b\right)$. Using convex ordering again yields that
$a$ and $b$ are the first negative resp. positive zero of $f_{R}$
. Hence $R$ is Loynes-regular.\end{proof}
\begin{lem}
\label{thm:barrieruniqueness}Let $\left(\sigma,\mu,\nu\right)$ fulfill
Assumption \ref{as:main assumption} and assume that there exists
$Q\in\mathcal{R}$ such that $\tau^{Q}$ solves $\operatorname{SEP}_{\sigma,\mu,\nu}$.
Then there also exists unique $\left(\mu,\nu\right)$-regular barrier
$R\in\mathcal{R}_{\mu,\nu}$ such that $\tau_{R}$ solves $\operatorname{SEP}_{\sigma,\mu,\nu}$.
\end{lem}
\begin{proof}
To see that $Q$ is $\left(\mu,\nu\right)$-equivalent to a $\left(\mu,\nu\right)$-regular
barrier just note that since $u_{\nu}\leq u_{\mu}$ (and we embed
by assumption) the continuous time-space process $\left(t\wedge\tau_{Q},X_{t\wedge\tau_{Q}}\right)$
does not enter $\left[0,\infty\right]\times N^{\mu,\nu}$, hence $R:=Q\cup\mathcal{N}^{\mu,\nu}$
is a also an element of $\mathcal{R}_{\mu,\nu}$ that solves $\operatorname{SEP}_{\sigma,\mu,\nu}$.

Suppose there are two $\left(\mu,\nu\right)$-regular barriers $B,C$,
each embedding $\nu$ (via $X$) with u.i.~stopping times $\tau_{B}$
and $\tau_{C}$ respectively. Then $\Gamma=B\cup C$ also embeds $\nu$
with the u.i. stopping time $\gamma=\min\left\{ \tau_{B},\tau_{C}\right\} $,
this statement is a straightforward extension of \cite[Proposition 4]{MR0292170}
to our setting (the proof only requires continuity of the paths).
Furthermore, since $\tau_{B}$ and $\tau_{C}$ are u.i. they are minimal
(see \cite[Proposition 2.2.2 (p23)]{Obloj2012}) then $\gamma$ is
also minimal, this in turn implies that $\gamma=\tau_{B}=\tau_{C}$.
It remains to show that $B,C$ and $B\cup C$ are the same (outside
$\mathcal{N}^{\mu,\nu}$ since in $\mathcal{N}^{\mu,\nu}$ this must
hold) one argues as in the proof of \cite[Lemma 2 (p215)]{MR0292170}
by proving that if $B\neq\Gamma$ then also $\tau_{B}\neq\tau_{\gamma}$. \end{proof}
%\begin{lem}
%\label{monroes lemma-1}Let $\left(\sigma,\mu,\nu\right)$ fulfill
%Assumption \ref{as:main assumption}. Let $R\in\mathcal{R}$ and assume
%$X_{\tau^{R}}\sim\nu$. Then there also exists a $R^{\prime}\in\mathcal{R}_{\mu,\nu}$
%such that $\tau^{R^{\prime}}\in\operatorname{SEP}_{\sigma,\mu,\nu}$.\end{lem}
%\begin{proof}
%By Dambis/Dubins--Schwarz there exists (after possible augmentation
%of the probability space; wlog assume also that $\mu,\nu$ are centered
%around $0$) a Brownian motion $W$ such that $X_{t}=W_{\left[X\right]_{t}}$.
%Since $\mathbb{E}X_{\tau^{R}}=0$ the stopping time $\sigma=\left[X\right]_{\tau^{R}}$
%fulfills the assumptions of \cite[Lemma 4]{monroe1972embedding} which
%in turn guarantees the existence of a Root barrier $R^{\prime\prime}$
%such that $X_{\tau^{R}\wedge\tau^{R^{\prime\prime}}}=W_{\sigma\wedge\tau^{R^{\prime\prime}}}\sim\nu$
%and $\left(X_{\tau^{R}\wedge\tau^{R^{\prime\prime}}\wedge t}\right)_{t}$
%is u.i., hence $R\cup R^{\prime\prime}\in\mathcal{R}$ solves $\operatorname{SEP}_{\sigma,\mu,\nu}$.
%Lemma \ref{thm:barrieruniqueness} then guarantees the existence of
%$R^{\prime}\in\mathcal{R}_{\mu,\nu}$ that solves the same embedding
%problem.
%\end{proof}

\subsection{Root's solution as a free boundary\label{sub:Root's-solution-asfreeboudnary}}

We have already seen in Theorem \ref{thm:supersolution} that every
solution of (\ref{eq:sep}) gives rise to a supersolution of an obstacle
PDE. The theorem below gives a complete characterization of (regular)
Root solutions. 
\begin{thm}
\label{thm:root as pde}Let $\left(\sigma,\mu,\nu\right)$ fulfill
Assumption \ref{as:main assumption}. Then the following are equivalent:
\begin{enumerate}
\item there exists $R\in\mathcal{R}_{\mu,\nu}$ such that $\tau^{R}=\inf\left\{ t>0:\left(t,X_{t}\right)\in R\right\} \in\operatorname{SEP}_{\sigma,\mu,\nu}$,
\item there exists a viscosity solution $u\in C\left(\left[0,\infty\right],\left[-\infty,\infty\right]\right)$
%that is of linear growth (uniformly in time)%
%\footnote{That is, $\sup_{\left(t,x\right)\in\left[0,\infty\right)\times\mathbb{R}}\frac{\left|u\left(t,x\right)\right|}{1+\left|x\right|}<\infty$.%
%} and 
decreasing (in time) of\textup{
\begin{equation}
\left\{ \begin{array}{rcl}
\min\left(u-u_{\nu},\partial_{t}u-\frac{\sigma^{2}}{2}\Delta u\right) & = & 0\text{ on }\left(0,\infty\right)\times\mathbb{R},\\
u\left(0,.\right) & = & u_{\mu}\left(.\right),\\
u\left(\infty,.\right) & = & u_{\nu}\left(.\right).
\end{array}\right.\label{eq:obsPDE}
\end{equation}
}
\end{enumerate}
\textup{\emph{Moreover,}}\textup{ 
\begin{equation}
R=\left\{ \left(t,x\right)\in\left[0,\infty\right]\times\left[-\infty,\infty\right]:u\left(t,x\right)=u_{\nu}\left(x\right)\right\} \text{ and }u\left(t,x\right)=-\mathbb{E}\left[\left|X_{\tau^{R}\wedge t}-x\right|\right].\label{eq:representations R and u}
\end{equation}
}\end{thm}
\begin{proof}
We first show that \textbf{(i) implies (ii)}: that is we have to show
that the function
\begin{equation}
u\left(t,x\right):=-\mathbb{E}\left[\left|X_{\tau^{R}\wedge t}-x\right|\right]\label{eq:rep-1}
\end{equation}
(identified with its limit $u_{\mu}$ resp.~$u_{\nu}$ as $t\rightarrow0$
resp. $\infty$) fulfills in viscosity sense
\begin{equation}
\left\{ \begin{array}{rcl}
u-u_{\nu} & \geq & 0\text{ on }\left(0,\infty\right)\times\mathbb{R},\\
\left(\partial_{t}-\frac{\sigma^{2}}{2}\Delta\right)u & \geq & 0\text{ on }\left(0,\infty\right)\times\mathbb{R},\\
u-u_{\nu} & = & 0\text{ on }R,\\
\left(\partial_{t}-\frac{\sigma^{2}}{2}\Delta\right)u & = & 0\text{ on }R^{c}.
\end{array}\right.\label{eq:inequalities}
\end{equation}
The first and second line in (\ref{eq:inequalities}) follow from
Theorem \ref{thm:supersolution}. To see that $u-u_{\nu}=0\text{ on }R$,
note that by Tanaka's formula 
\[
u\left(t,x\right)=u_{\mu}\left(x\right)-\mathbb{E}\left[L_{t\wedge\tau^{R}}^{x}\right]\,\,\,\forall\left(t,x\right)
\]
and letting $t\rightarrow\infty$ gives 
\[
u_{\nu}\left(x\right)=u_{\mu}\left(x\right)-\mathbb{E}\left[L_{\tau^{R}}^{x}\right].
\]
Subtracting from the above yields
\[
u\left(t,x\right)-u_{\nu}\left(x\right)=\mathbb{E}\left[L_{\tau^{R}}^{x}-L_{t\wedge\tau^{R}}^{x}\right]
\]
and therefore it is sufficient to show that $L_{\tau^{R}}^{x}-L_{t\wedge\tau^{R}}^{x}=0$
for $\left(t,x\right)\in R$, $\mathbb{P}$-a.s. To see this simply
write 
\[
L_{\tau^{R}}^{x}-L_{t\wedge\tau^{R}}^{x}=\left(L_{\tau^{R}}^{x}-L_{t}^{x}\right)1_{t<\tau^{R}}
\]
and note that the right hand side can only be strictly positive if
the process $\left(X_{s\wedge\tau^{R}}\right)_{s\geq t}$ crosses
the line $\left\{ \left(s,x\right):s\in\left[t,\tau^{R}\right]\right\} $.
However, since $R$ is a Root barrier and $\left(t,x\right)\in R$
we have that $\left\{ \left(s,x\right):s\in\left[t,\tau^{R}\right]\right\} \subset R$
and since $\tau^{R}$ is the first hitting time of $R$ this event
is a null event. It now only remains to show\textbf{ }
\[
\left(\partial_{t}-\frac{\sigma^{2}}{2}\Delta\right)u=0\text{ on }R^{c}.
\]
and to do this we argue again via stability as in Theorem
\ref{thm:supersolution}. Therefore define $u^{n}$,$u_{\mu}^{n}$
and $u_{\nu}^{n}$ as well as $I_{n}$ and $II_{n}$ exactly as in
Theorem \ref{thm:supersolution}. From Lemma \ref{lem:jets for u-1}
it follows that if $\partial_{t-}u^{n}\left(t,x\right)<\partial_{t+}u^{n}\left(t,x\right)$
then $\mathcal{P}_{\mathcal{O}}^{2,+}u^{n}\left(t,x\right)=\emptyset$
(in which case we are done) and if $\partial_{t}u^{n}\left(t,x\right)=\partial_{t-}u^{n}\left(t,x\right)=\partial_{t+}u^{n}\left(t,x\right)$
then $\forall\left(a,p,m\right)\in\mathcal{P}_{\mathcal{O}}^{2,+}u^{n}\left(t,x\right)$,
$a=\partial_{t}u^{n}\left(t,x\right)$ and $m\geq\Delta u^{n}\left(t,x\right)$.
Hence, in the latter case we have $\forall\left(a,p,m\right)\in\mathcal{P}_{\mathcal{O}}^{2,+}u^{n}\left(t,x\right)$
that 
\begin{eqnarray*}
a-m & \leq & \partial_{t}u^{n}\left(t,x\right)-\frac{\sigma^{2}\left(t,x\right)}{2}\Delta u^{n}\left(t,x\right).
\end{eqnarray*}
As in Theorem \ref{thm:supersolution}, we see that $u^{n}$ is a
subsolution of
\[
\left\{ \begin{array}{rcl}
\partial_{t}u-\frac{\sigma^{2}}{2}\Delta u-\frac{1}{2}\left(I_{n}+II_{n}\right) & = & 0\text{ on }\left(0,\infty\right)\times\mathbb{R},\\
u\left(0,.\right) & = & u_{\mu}^{n}\left(.\right).
\end{array}\right.
\]
In Theorem \ref{thm:supersolution} we have already shown that $II_{n}\rightarrow0$
locally uniformly as $n\rightarrow\infty$ and we now show that also
$I^{n}\rightarrow0$ locally uniformly on $R^{c}$: since $R$ is
a Root barrier we have
\[
\left(\tau_{R}+r,X_{\tau_{R}}\right)\in R\,\,\,\forall r\geq0,
\]
hence if $\left(t,x\right)\in R^{c}$ and $t\geq\tau_{R}$ then $X_{\tau_{R}}\neq x$.
Therefore 
\[
\lim_{n\rightarrow\infty}\sup_{\left(t,x\right)\in K}\Delta\psi^{n}\left(X_{\tau_{R}}-x\right)1_{t\geq\tau_{R}}=0\,\,\,\text{ for every compact }K\subset R^{c}
\]
which is enough to conclude that $I_{n}$ converges locally uniformly
on $R^{c}$ to $0$, i.e.~for every compact $K\subset R^{c}$ 
\[
\lim_{n}\sup_{\left(t,x\right)\in K}I_{n}\left(t,x\right)\leq\left|\sigma^{2}\right|_{\infty,K}\lim_{n}\mathbb{E}\left[\sup_{\left(t,x\right)\in K}\Delta\psi^{n}\left(X_{\tau_{R}}-x\right)1_{t\geq\tau_{R}}\right]=0.
\]
The stability results and the restatement for parabolic PDEs of Proposition
4.3, Lemma 6.1 and Remark 6.4 found in the User's guide \cite{MR1118699UserGuide}
imply that $u$ is a subsolution of

\[
\left\{ \begin{array}{rcl}
\left(\partial_{t}-\frac{\sigma^{2}}{2}\Delta\right)u & = & 0\text{ on }R^{c},\\
u\left(0,.\right) & = & u_{\mu}^{n}\left(.\right).
\end{array}\right.
\]
Putting the above together shows that $u$ is indeed a viscosity solution
of the obstacle problem (\ref{eq:obsPDE}). To see that $u$ is of
linear growth, recall that by the above $u\left(t,x\right)\in\left[u_{\mu}\left(x\right),u_{\nu}\left(x\right)\right]$,
hence $\left|u\left(t,x\right)\right|\leq\left|u_{\mu}\left(x\right)\right|+\left|u_{\nu}\left(x\right)\right|$.
Since $u_{\mu}$ and $u_{\nu}$ are of linear growth (see e.g.~\cite[Section 3.2]{Obloj2012},\cite[Proposition 2.1]{hirsch2012new},\cite[Proposition 4.1]{BeiglboeckJuillet2012})
we have shown that 
\begin{equation}
\sup_{\left(t,x\right)\in\left[0,\infty\right)\times\mathbb{R}}\frac{\left|u\left(t,x\right)\right|}{1+\left|x\right|}<\infty.\label{eq:linear growth}
\end{equation}
This allows to use our comparison result, Theorem \ref{thm:compforOSB},
to conclude that $u$ is not only a solution but the unique viscosity
solution of linear growth. Thus we have shown that (i) implies (ii)
and that the second equality in (\ref{eq:representations R and u})
holds.

We now show that \textbf{(ii) implies (i)}. First note that since $u_\nu \leq u \leq u_\mu$, $u$ has linear growth in space uniformly in time. Now set 
\[
R:=\left\{ \left(t,x\right):u\left(t,x\right)=u_{\nu}\left(x\right)\right\} 
\]
and write $R$ as $R=\bigcup_{t\geq0}\left\{ t\right\} \times R_{t}$ for closed sets $R_{t}\subset\left[0,\infty\right]$. 
Note that since $u$ is decreasing in time, $R$ is actually a barrier, and it is clearly $(\mu,\nu)$-regular. 
To see that the free boundary $R$ embeds $\nu$, we introduce $R^{-}:=\bigcup_{t\geq0}\left\{ t\right\} \times R_{t-}$ where $R_{t-}:=\bigcup_{s<t}R_{s}$ and denote with $\nu_{R}$ and $\nu_{R^-}$ the distributions of $X_{\tau^{R}}$ and $X_{\tau^{R^{-}}}$ and with $u_{\nu_R}$ and $u_{\nu_R^-}$ the potential functions. 
Since $R^-\subset R$, we already know that  $u_{\nu_{R^-}}\leq u_{\nu_R}$ and we now show that  
\begin{equation}
u_{\nu_{R^-}}\leq u_{\nu}\leq u_{\nu_R}.\label{eq:interval for potential}
\end{equation}
We then argue that the above have to be equalities which shows the desired embedding.
To this end, consider for each $\epsilon >0$ the shifted barriers
\[
R_{\epsilon}=\bigcup_{t\geq\epsilon} \left\{t \right\} \times R_{t-\epsilon}\text{ and }R^{\epsilon}=\bigcup_{t\geq0} \left\{t \right\} \times R_{t+\epsilon}
\]
and denote their corresponding hitting times by $t\mapsto\left(t,X_{t}\right)$
with $\tau^{\epsilon}$,$\tau_{\epsilon}$ and the corresponding potential
functions with $u^{\epsilon}\left(t,x\right):=-\mathbb{E}\left[\left|X_{\tau^{\epsilon}\wedge t}-x\right|\right]$,$u_{\epsilon}(t,x):=-\mathbb{E}\left[\left|X_{\tau_{\epsilon}\wedge t}-x\right|\right]$.
Note that 
\[
R_{\epsilon}\subset R^- \subset R\subset R^{\epsilon}.
\]
Let us first prove that 
\begin{equation}
\lim_{\epsilon \to 0} \tau_\epsilon = \tau_{R^-}, \;\;\; \lim_{\epsilon \to 0} \tau^\epsilon = \tau_{R}.
\label{eq:st}
\end{equation}
For the first equality, note that if $t \leq \tau_\epsilon$ then $f_R(X_s) > s -\epsilon$ for all $s \geq t$. Hence if $t \leq \inf{\epsilon} \tau_\epsilon$, then for all $s \leq t$, $f_R(X_s) \geq s$ for all $s \leq t$, i.e. $(s,X_s) \notin R^-$. It follows that $\lim_{\epsilon} \tau_\epsilon \leq \tau_{R^-}$, and the reverse inequality is obvious since $R_\epsilon \subset R^{-}$. For the second inequality, passing to the limit in the relation $f_R(X_{\tau^\epsilon}) \leq \tau^\epsilon + \epsilon$ and using lower-semicontinuity of $f_R$ yields $f_R(X_{\lim_\epsilon \tau^\epsilon}) \leq \lim_\epsilon \tau^\epsilon$, i.e. $\lim_\epsilon \tau^\epsilon \geq \tau_R$, and again the reverse inequality is obvious.

We now claim that
\begin{equation}
u_{\epsilon} \leq u \leq u^{\epsilon},
\label{eq:IneqEps}
\end{equation}
which by \eqref{eq:st} already shows (\ref{eq:interval for potential}). 

Let us prove the first inequality in \eqref{eq:IneqEps}. It will follow from a simple application of viscosity comparison. Indeed, let $v$ $=$ $u_\epsilon-u$, we now show that it satisfies in viscosity sense 
\begin{equation}
\label{eq:viscw}
\partial_t v - \left(\frac{\sigma^2}{2}\partial_{xx} v\right)_+ \leq 0.
\end{equation}
Indeed, let $(t,x)$ be in $R_\epsilon$. Then for all $s \geq t$, one has $u_\epsilon(s,x)=u_\epsilon(t,x)$ by the arguments from (i) $\Rightarrow$ (ii), whereas by definition of $R$, $u(s,x) = u(t,x)$ for all $s \in [t-\epsilon,\infty)$. Since in addition $u_\epsilon$ is nonincreasing in $t$, it follows that $\partial_{t+}v\left(t,x\right) \leq 0 = \partial_{t-}v\left(t,x\right)$. Hence by Lemma \ref{lem:jets for u-1}, $w$ satisfies $\partial_t w \leq 0$ in viscosity sense at $(t,x)$. And one has $(\partial_t - \frac{\sigma^2}{2}\partial_{xx})v = 0$ on $R^c$, again by respectively definition of $R$ and the arguments from (i) $\Rightarrow$ (ii). We have thus proved that on the whole space, one has $\min\left[ \partial_t v, \partial_t v - \frac{\sigma^2}{2}\partial_{xx} v\right] = \partial_t v - \left(\frac{\sigma^2}{2}\partial_{xx} v\right)_+\leq 0$, and therefore $v$ $\leq$ $0$ by comparison. The proof of the second inequality of \eqref{eq:IneqEps} is essentially the same.

This finishes the proof of (\ref{eq:interval for potential}). Now note that since one-point sets are regular for our one-dimensional diffusion $X$, one has $\tau_R = \tau_{R^-}$ a.s., so that the inequalities in (\ref{eq:interval for potential}) are actually equalities, which proves that the hitting time of $R$ embeds $\nu$.
To finish the proof, it remains to show that $X^{\tau_{R}}$ is uniformly integrable. But this is immediate since the family of laws $\left(\mathbb{P} \circ X_{t \wedge \tau_R}^{-1}\right)_{t\geq 0}$ is dominated in convex order by $\nu$, and is therefore u.i. by de La Vallee Poussin's theorem.
 \end{proof}

\begin{rem}
Work of Dupire and Cox--Wang \cite{Dupire2005,CoxWang2011} showed
that if classic existence results \cite{MR0238394,MR0445600} apply,
the barrier can be calculated via a PDE. What is new here is that
Theorem \ref{thm:root as pde} provides a complete characterization;
especially, it allows to the infer existence of a Root solution from
the existence of a PDE solution. This also applies to the time-inhomogenous
case where these classic approaches \cite{MR0238394,MR0445600} break down.
As we will see below, together with Theorem \ref{thm:supersolution}
it recovers and extends the minimizing property of barrier solutions.
\end{rem}
Theorem \ref{thm:root as pde} allows to infer the existence of a
Root solution for $\operatorname{SEP}_{\sigma,\mu,\nu}$ via the existence
of viscosity solutions for $\left(\sigma,\mu,\nu\right)$ in the 
%special case where $\sigma=\sigma(x)$ does not depend on $t$.
full
generality of Assumption (\ref{as:main assumption}). 
%Although nonlinear,
%the existence of a viscosity solution for above PDE follows from a
%simple application of Perron's method, see \cite[Chapter 4]{MR1118699UserGuide}.
%We state all this as a simple corollary.
\begin{cor}[Existence of Root solutions]
\label{cor:existence root}Let $\left(\sigma,\mu.\nu\right)$ fulfill
Assumption \ref{as:main assumption} 
%and further assume that $\sigma$ is not time-dependent
. Then (i) resp.~(ii) in Theorem \ref{thm:root as pde}
hold. Especially, there exists a unique $R\in\mathcal{R}_{\mu,\nu}$
such that $\tau^{R}\in\operatorname{SEP}_{\sigma,\mu,\nu}$ and $R$
is the free boundary (\ref{eq:representations R and u}) of the obstacle
PDE (\ref{eq:obsPDE}). \end{cor}
\begin{proof}
Existence of a viscosity solution $u$ to
\begin{equation*}
\left\{ \begin{array}{rcl}
\min\left(u-u_{\nu},\partial_{t}u-\frac{\sigma^{2}}{2}\Delta u\right) & = & 0\text{ on }\left(0,\infty\right)\times\mathbb{R},\\
u\left(0,.\right) & = & u_{\mu}\left(.\right)
\end{array}\right.
\end{equation*}
 follows from standard results, for
example by penalization and Perron's method (see \cite{el1997reflected}). Hence it only remains to prove that the solution $u$ is decreasing in time, and satisfies $u(\infty,\cdot) = u_\nu$.

\underline{1) $u$ is decreasing in time :} 

We first prove it in the case where $\sigma=\sigma(x)$ does not depend on $t$. Define for $h>0$, the function $u^{h}\left(t,x\right):=u\left(t+h,x\right)$.
Since $u_{\mu}$ is concave, it is a supersolution of (\ref{eq:obsPDE}), and since 
$u^{h}\left(0,x\right)\equiv u\left(h,x\right)\leq u_{\mu}\left(x\right)$
and $u^{h}$ solves the very same PDE as $u$, it follows by comparison
that $u^{h}\left(t,x\right)\leq u\left(t,x\right)$. Note in addition that since $\frac{\sigma^2}{2} \partial_xx u \leq \partial_t u$ (in viscosity sense), the fact that $u$ is decreasing in $t$ is easily seen to imply that $u$ is concave in $x$. Now consider $\sigma$ piecewise-constant in $t$. Then by iterating the above argument, one gets that $u$ is decreasing in $t$ and concave in $x$. To obtain the general result, approximate the continuous function $\sigma(t,x)$ by a sequence $\sigma^\epsilon$ each of these being piecewise-constant in $t$. Then the corresponding solutions $u^\epsilon$ converge locally uniformly to $u$, which therefore has the same monotonicity and concavity properties. (Note that since the $\sigma^\epsilon$ are not continuous functions, one needs to use the existence/uniqueness/stability results for viscosity solutions with discontinuous coefficients, e.g.~\cite{MR1167420,MR2246335})

\underline{2) $u(\infty,\cdot) = u_\nu$ :}

Let $O=\left\{ u(\infty,\cdot) > u_\nu\right\}$, and assume that $(a,b) \subset O$. Then by local ellipticity $\sigma \geq c>0$ on $(a,b)$, and one has on $[0,\infty) \times (a,b)$, $(\partial_t - \frac{c^2}{2} \partial_{xx})u  \leq (\partial_t - \frac{\sigma^2}{2} \partial_{xx})u = 0$ (using concavity of $u$ for the first inequality). Hence for each $t \geq 0$, $u$ is dominated on $[t,\infty] \times [a,b]$ by the solution $v$ to
\begin{equation*}
\left\{ \begin{array}{ll}
	(\partial_t - \frac{c^2}{2} \partial_{xx}) v = 0 & \mbox{ on }(t,\infty) \times (a,b), \\
	v(t,x) = u(t,x), & \forall x \in (a,b) \\
	v(s,a) = u(t,a), \;\;\; v(s,b) = u(t,b),  & \forall s \in [t,\infty].
\end{array}
\right.
\end{equation*}
But by standard computations, $v(\infty,\cdot)$ is the linear interpolation between $u(t,a)$ and $u(t,b)$, so that by comparison and letting $t \to \infty$, we obtain
$$u(\infty, \delta a + (1-\delta) b) \leq \delta u(\infty, a) + (1-\delta) u(\infty,b), \;\; \forall \delta \in [0,1],$$
i.e. $u(\infty,\cdot)$ is convex on any connected component of $O$. Now let $(c,d)$ be a connected component of $O$. If $-\infty < c <d < \infty$, then one has $u(\infty,c) = u_\nu(c)$ and $u(\infty,d) = u_\nu(d)$. But since $u_\mu$ is concave, it must necessarily dominate the convex function $u(\infty,\cdot)$ on $(c,d)$, contradicting the fact that $(c,d) \subset O$. Similarly when $c$ or $d$ is infinite one gets a contradiction using $\lim_{x \to \infty} (u_\mu-u_\nu)(x) = 0$. Hence $O = \emptyset$, and $u(\infty,\cdot) \equiv u_\nu$.
\end{proof}
Moreover, the proof of Theorem \ref{thm:root as pde} also shows
regularity properties about $u$. They have a intuitive explanation
by their representation as potential functions so we record them as
a corollary.
\begin{cor}
Let $\left(\sigma,\mu.\nu\right)$ fulfill Assumption \ref{as:main assumption}.
Then the viscosity solution $u$ from Theorem \ref{thm:root as pde}
fulfills
\begin{enumerate}
\item for every $x\in\mathbb{R}$ $t\mapsto u\left(t,x\right)$ is non-increasing
and 
\[
u_{\nu}\left(x\right)\leq u\left(t,x\right)\leq u_{\mu}\left(x\right)\text{ \,\,}\forall\left(t,x\right),
\]

\item $u\lvert_{R}=u_{\nu}\lvert_{R}$,
\item $u$ is Lipschitz in space (uniformly in time), 
\[
\sup_{t\in\left[0,\infty\right)}\sup_{x\neq y}\frac{\left|u\left(t,x\right)-u\left(t,y\right)\right|}{\left|x-y\right|}<\infty.
\]

\end{enumerate}
\end{cor}

\subsection{Root's solution as a minimizer\label{sub:Root as minimizer}}

After Root \cite{MR0238394} proved the existence of a barrier solution
for the Brownian case, Rost \cite{MR0445600} used potential theoretic
methods to show that Root's solution minimizes the residual expectation
\[
\mathbb{E}\left[\tau-\tau\wedge t\right]=\int_{0}^{t}\mathbb{P}\left(\tau>s\right)ds\,\,\forall t\geq0.
\]
(As is well known from old work of Dinge \cite{dinges1974stopping},
minimizing above residual expectation implies that $\tau$ also minimizes
$\mathbb{E}\left[f\left(\tau\right)\right]$ for $f\geq0$, convex)%b
\footnote{Applied with $f\left(x\right)=x^{2}$, Rost thereby proved a conjecture
made earlier by Kiefer \cite{kiefer1972skorohod}, namely that Root's
solution minimizes the variance. This property is the one that makes
Root's solution give lower bounds on options on variance. Though strictly
speaking, Rost \cite{MR0445600} proved Kiefer's conjecture only for
measures with bounded support as pointed out by himself \cite["Technical Remark" at the bottom of page 3]{MR0445600}.%
}. The viscosity PDE characterization of Theorem \ref{thm:root as pde}
now allows to give a very short proof of the minimizing property of
the Root barrier via our parabolic comparision result, Thereom \ref{thm:compforOSB}.
It immediately covers the time-inhomogenous and degenerate elliptic
case (thereby generalizing Rost's approach \cite{MR0445600}) and
is already for the simple Brownian case, $\sigma\equiv1$ and $\mu=\delta_{0}$,
the shortest minimality proof that we are aware of.
\begin{thm}
\label{thm:root as minimizer}Let $\left(\sigma,\mu,\nu\right)$,
$u$ and $R$ be as in Theorem \ref{thm:root as pde}. 
\begin{enumerate}
\item The potential function of the Root solution is a minimizer, that is
\begin{equation}
u\left(t,x\right)=\operatorname{argmin}u^{\tau}\left(t,x\right)\,\,\forall\left(t,x\right)\in\left(0,\infty\right)\times\mathbb{R}\label{eq:u minimizes}
\end{equation}
where $u^{\tau}\left(t,x\right)\equiv-\mathbb{E}\left[\left|X_{\tau\wedge t}-x\right|\right]$.
\item If we additionally assume that $\mu,\nu$ have second moments, then
Root's solution minimizes the residual expectation, that is
\begin{eqnarray*}
\tau^{R} & = & \operatorname{argmin}\mathbb{E}\left[\int_{\tau\wedge t}^{\tau}\sigma^{2}\left(r,X_{r}\right)dr\right]\,\,\,\forall t\geq0
\end{eqnarray*}
where $\tau^{R}$ is the hitting time of $R$.
\end{enumerate}
In both statements above, $\operatorname{argmin}$ is taken over $\tau\in\operatorname{SEP}_{\sigma,\mu,\nu}$. \end{thm}
\begin{proof}
 From Theorem \ref{thm:supersolution} we know that every $u^{\tau}$
is a supersolution of the obstacle PDE (\ref{eq:obsPDE}) and from
Theorem \ref{thm:root as pde} we know that $u^{\tau_{R}}$ is a solution
of the obstacle PDE (\ref{eq:obsPDE}). Using our parabolic comparison
result, Theorem \ref{thm:compforOSB}, for the supersolution $u^{\tau}$
and subsolution $u^{\tau_{R}}$ shows (\ref{eq:u minimizes}).

To see the second claim, note that by Ito 
\[
\mathbb{E}X_{\tau}^{2}-\mathbb{E}X_{t\wedge\tau}^{2}=\mathbb{E}\left[\int_{\tau\wedge t}^{\tau}\sigma^{2}\left(r,X_{r}\right)dr\right]
\]
and since $\mathbb{E}X_{\tau}^{2}=\int x^{2}\mu\left(dx\right)=\operatorname{const}$
we conclude that this is equivalent to the statement that the Root
stopping time $\tau^{R}$ maximises 
\[
\mathbb{E}X_{t\wedge\tau}^{2}=\mathbb{E}\left[\int_{\mathbb{R}}L_{t\wedge\tau}^{x}dx\right]=-\int_{\mathbb{R}}\left(u^{\tau}\left(t,x\right)-u^{\tau}\left(0,x\right)\right)dx.
\]
Here $\left(L_{t}^{x}\right)_{t,x}$ denotes local time $X$. Hence
it is sufficient to show that the Root stopping time $\tau^{R}$ minimises
$u^{\tau}\left(t,x\right)$ pointwise, i.e.~that for all $\tau\in\operatorname{SEP}_{_{\left(1,\mu,\nu\right)}}$
we have 
\[
-\mathbb{E}\left[\left|W_{\tau\wedge t}-x\right|\right]\equiv u^{\tau}\left(t,x\right)\geq u^{\tau_{R}}\left(t,x\right)\equiv-\mathbb{E}\left[\left|W_{\tau_{R}\wedge t}-x\right|\right]\,\,\,\forall\left(t,x\right).
\]
However, this follows by (i).
\end{proof}

\subsection{\label{sub:bsde}Root's solution via RBSDEs}

Using Theorem \ref{thm:root as pde} we can give another characterization
of the Root solution via Reflected FBSDEs by using \cite{el1997reflected}.
Our main interest is that it gives rise to Monte-Carlo methods to
solve for the barrier. However, it also clarifies further how the Root solution
is naturally linked to a stopping problem\footnote{We point the reader to \cite{BuckdahnHuangLi2012} for
a finer analysis on the regularity of RFBSDE solutions and such connections; the analysis there though does not immediately
cover the current case due to unboundedness of coefficients.}.
\begin{cor}[RBSDE representation]
\label{prop:existenceOBS}Let $\left(\sigma,\mu,\nu\right)$ fulfill
Assumption \ref{as:main assumption}. Then
\begin{enumerate}
\item there exists a unique $R\in\mathcal{R}_{\mu,\nu}$ such that $\tau^{R}=\inf\left\{ t>0:\left(t,X_{t}\right)\in R\right\} \in\operatorname{SEP}_{\sigma,\mu,\nu}$,
\item and for every $T>0$ 
\[
R\lvert_{\left[0,T\right]\times\left[-\infty,\infty\right]}=\left\{ \left(t,x\right):Y_{T-t}^{T-t,x}=u_{\nu}\left(x\right)\right\} 
\]
 where $Y$ denotes the backward dynamics of the solution $\left(N,Y,Z,K\right)$
of the RBSDE%
\footnote{$\left\{ \mathcal{G}_{s}^{t},t\leq s\leq T\right\} $ denots the natural
filtration of a Brownian motion $\left\{ W_{s}-W_{t},\, t\leq s\leq T\right\} $
augmented with the null sets of $\mathcal{G}$. The quadruple $\left(N,Y,Z,K\right)$
is $\mathcal{G}$-adapted and $\left(K_{s}^{t,x}\right){}_{s\in\left[t,T\right]}$
is an increasing and continuous process verifying $K_{t}^{t,x}=0$.
Note $\left(N,Y,Z,K\right)$ does not have to be defined on the same
probability space as our forward martingale $dX=\sigma\left(t,X_{t}\right)dB_{t}$
but with slight abuse of notation we denote the expectation still
with $\mathbb{E}$.%
} 
\begin{equation}
\left\{ \begin{array}{rcl}
N_{s}^{t,x} & = & x+\int_{t}^{s}\sigma\left(T-r,N_{r}^{t,x}\right)dW_{r},\\
Y_{s}^{t,x} & = & u_{\mu}\left(N_{T}^{t,x}\right)+K_{T}^{t,x}-K_{s}^{t,x}-\int_{s}^{T}Z_{r}^{t,x}dW_{r},\\
Y_{s}^{t,x} & \geq & u_{\nu}\left(N_{s}^{t,x}\right),\, t<s\leq T\text{ and }\int_{t}^{T}\left(Y_{s}^{t,x}-u_{\nu}\left(N_{s}^{t,x}\right)\right)dK_{s}^{t,x}=0.
\end{array}\right.\label{eq:RBSDEsubclass-1}
\end{equation}

\end{enumerate}
Moreover, the solution $u$ of the obstacle problem (\ref{eq:obsPDE})
solves the stopping problem
\[
u\left(T-t,x\right)=\sup_{\tau\in\mathcal{T}_{t,T}}\mathbb{E}\left[u_{\nu}\left(N_{\tau}^{t,x}\right)1_{\tau<T}+u_{\mu}\left(N_{\tau}^{t,x}\right)1_{\tau=T}\right]
\]
where $\mathcal{T}_{t,T}=\left\{ \tau:\,\tau\text{ is a \ensuremath{\mathcal{G}}-stopping time and }\tau\in\left[t,T\right]\text{ a.s.}\right\} $.
\end{cor}
$Z$ guides the evolution of $Y$ and $K$ via the Itô integral so
that $Y$ hits the random variable $u_{\mu}\left(N_{T}\right)$ at
horizon time $T$. Note that $Y,K,Z$ are $\left(\mathcal{G}_{t}\right)$-adapted,
nonetheless, $u_{\mu}\left(N_{T}\right)$ is attained at $t=T$. The
process $K$ ensures that $Y$ does not go below the barrier $u_{\nu}$;
it pushes $Y$ upwards whenever $Y$ touches and tries to go below
the barrier $u_{\nu}$, else it remains inactive (that is constant)
--- $K$ is minimal in this sense. Above interpretation as optimization
problem on finite time horizon $T<\infty$ is a special case of 
\begin{align}
K_{T}^{t,x}-K_{s}^{t,x} & =\sup_{s\leq u\leq T}\left(u_{\mu}\left(N_{T}^{t,x}\right)-\int_{u}^{T}Z_{r}^{t,x}W_{r}-u_{\nu}\left(N_{u}^{t,x}\right)\right)^{-},\nonumber \\
Y_{s}^{t,x}=w\left(s,N_{s}^{t,x}\right) & =\sup_{\tau\in\mathcal{T}_{s}}\mathbb{E}\left[u_{\nu}\left(N_{\tau}^{t,x}\right)1_{\tau<T}+u_{\mu}\left(N_{T}^{t,x}\right)1_{\tau=T}\Big|\mathcal{G}_{s}^{t}\right]\nonumber \\
 & =\sup_{\tau\in\mathcal{T}_{s}}\mathbb{E}\left[\left(u_{\nu}\left(N_{\tau}^{t,x}\right)-u_{\mu}\left(N_{\tau}^{t,x}\right)\right)1_{\tau<T}+u_{\mu}\left(N_{\tau}^{t,x}\right)\Big|\mathcal{G}_{s}^{t}\right]\label{eq:snellenvelopsup}
\end{align}
applied with $s=t$; here $\mathcal{T}_{s}:=\left\{ \tau\in\mathcal{T}:s\leq\tau\leq T\right\} $.
Following the theory of Snell envelopes, $Y$ is simply the smallest
supermartingale which dominates the sum inside the expectation. Lastly,
the optimal stopping time solving the above optimization problem (for
$Y_{s}^{t,x}$) is known to be 
\[
D_{s}^{t,x}:=\inf\left\{ s\leq r\leq T:Y_{r}^{t,x}=u_{\nu}\left(N_{r}^{t,x}\right)\right\} 
\]
 with $D_{s}^{t,x}=T$ if $Y_{r}^{t,x}>u_{\nu}\left(N_{r}^{t,x}\right)$
for all $s\leq r\leq T$. 
\begin{rem}
This further clarifies the connection to optimal stopping that can
be seen from the PDE (see also \cite[Remark 4.4]{CoxWang2011}). However
note that the optimization problem is rather non-standard due to the
time reversal and that many embeddings require us to include $T=\infty$
for which the time reversal and RBSDE representation breaks down (at
least for time-inhomogenous $\sigma$).
\end{rem}

\begin{rem}
The following \emph{formal }argument gives at least an intuition why
RBSDE and obstacle PDEs are in a similar relation as SDEs and linear
PDEs: suppose a sufficiently regular solution $w$ of (\ref{eq:OBS-reflectedFBSDE})
exists. Via Itô's formula it follows that
\[
Y_{s}^{t,x}:=w\left(s,N_{s}^{t,x}\right),\, Z_{s}^{t,x}:=\left(\bar{\sigma}\nabla_{x}w\right)\left(s,N_{s}^{t,x}\right)\text{ and }K_{s}^{t,x}:=\int_{t}^{s}\left(-\partial_{t}w-\frac{\overline{\sigma}^{2}}{2}\Delta w\right)\left(r,N_{r}^{t,x}\right)dr.
\]
solves the RFBSDE. The last condition in (\ref{eq:RBSDEsubclass-1})
then reads as 
\[
\int_{t}^{T}\left(Y_{r}^{t,x}-u_{\nu}\left(N_{r}^{t,x}\right)\right)dK_{r}^{t,x}=0\text{ iff }\int_{t}^{T}\left[\left(w-u_{\nu}\right)\left(-\partial_{t}w-\frac{\overline{\sigma}^{2}}{2}\Delta w\right)\right]\left(r,N_{r}^{t,x}\right)dr=0
\]
and the rhs explains the form of the PDE fulfilled by $w$.\end{rem}
\begin{proof}[Proof of Corollary \ref{prop:existenceOBS}]
In view of Theorem \ref{thm:root as pde} we only need to show that
there exists a quadruple $\left(N,Y,Z,K\right)$ that fulfills (\ref{eq:RBSDEsubclass-1})
and that $u\left(t,x\right):=Y_{T-t}^{T-t,x}$ yields a viscosity
solution with linear growth uniform in time.

\emph{Existence \& uniqueness in $\left[0,T\right]\times\mathbb{R}$,
$T<\infty$:} by time reversion of (\ref{eq:obsPDE}) shows that it
is enough to deal with 
\begin{equation}
\left\{ \begin{array}{rcl}
\min\left(w\left(t,x\right)-u_{\nu}\left(x\right),\left(-\partial_{t}-\frac{\overline{\sigma}^{2}}{2}\Delta\right)w\left(t,x\right)\right) & = & 0,\,\left(t,x\right)\in\mathcal{O}_{T}=\left(0,T\right)\times\mathbb{R}\\
w\left(T,x\right) & = & u_{\mu}\left(x\right),\, x\in\mathbb{R}.
\end{array}\right.\label{eq:OBS-reflectedFBSDE}
\end{equation}
where $\overline{\sigma}\left(t,x\right):=\sigma\left(T-t,x\right)$.
Continuity, existence and uniqueness of the viscosity solution $w$
follows from Lemma 8.4, Theorems 8.5 and 8.6 in \cite{el1997reflected}
respectively. The linear growth of $w$ in its spatial variable follows
from standard manipulations for RBSDEs. \cite[Proposition 3.5]{el1997reflected}
applied to the RFBSDE setting above (i.e.~using $\left(t,x\right)\mapsto\sigma\left(T-t,x\right)$
due to the time reversion argument) yields the existence of a constant
$\overline{k}_{T}>0$ such that $\forall\left(t,x\right)\in\left[0,T\right]\times\mathbb{R}$
\[
\left|Y_{t}^{t,x}\right|{}^{2}\leq\mathbb{E}\left[\sup_{s\in\left[t,T\right]}\left|Y_{s}^{t,x}\right|{}^{2}\right]\leq\overline{k}_{T}\left(\mathbb{E}\left[\left|u_{\mu}\left(N_{T}^{t,x}\right)\right|{}^{2}\right]+\mathbb{E}\left[\sup_{s\in\left[t,T\right]}\left|u_{\nu}^{+}\left(s,N_{s}^{t,x}\right)\right|{}^{2}\right]\right)\leq c_{T}\left(1+\left|x\right|{}^{2}\right)
\]
with $u_{\nu}^{+}:=\max\left\{ 0,u_{\nu}\right\} $ and where the
last inequality follows from the linear growth assumptions on $u_{\mu}$
and $u_{\nu}$ along with standard SDE estimates: $\sup_{t\in[0,T]}\mathbb{E}\left[\left|N_{T}^{t,x}\right|{}^{2}\right]\leq\hat{c}_{T}\left(1+\left|x\right|{}^{2}\right)$
(see e.g.~\cite[Equation (4.6)]{el1997reflected}). The solution
to (\ref{eq:obsPDE}) now follows from \cite{el1997reflected}. 

Above estimate for the linear growth in the spatial variable can be
made sharper in the sense that the constant $c_{T}$ is independent
of $T$. This follows via comparison results for RFBSDE (see \cite[Theorem 4.1 (p712)]{el1997reflected}).
Since $u_{\nu}\leq u_{\mu}\leq0$, i.e.~the terminal condition $u_{\mu}$is
non-positive, the component $Y$ is also non-positive. On the other
hand, the solution can not go below the barrier and hence $\left|Y_{T-t}^{T-t,x}\right|\leq\left|u_{\nu}\left(x\right)\right|\leq1+\left|x\right|$. 

\emph{Existence \& uniqueness in $\left[0,\infty\right)\times\mathbb{R}$:
}For any $T,T^{\prime}>0$, $Y_{T-t}^{T-t,x}$ and $Y_{T^{\prime}-t}^{T^{\prime}-t,x}$
coincide on $\left[0,T\wedge T^{\prime}\right)\times\mathbb{R}$.
Hence, we define a function $w\in C\left(\left[0,\infty\right),\mathbb{R}\right)$
by letting $w\left(t,x\right):=Y_{T-t}^{T-t,x}$ for arbitrary chosen
$T>t$. Then $u\left(t,x\right):=w\left(T-t,x\right)$ is the unique
viscosity solution of linear growth uniformly in time of $\min\left(u-u_{\nu},\partial_{t}-\frac{\sigma^{2}}{2}\Delta u\right)=0$,
$u\left(0,\cdot\right)=u_{\mu}\left(\cdot\right)$ (via our comparison
Theorem \ref{thm:compforOSB}). In Corrollary \ref{cor:existence root} we have already
shown that under Assumption \ref{as:main assumption} the solution
$u$ must be decreasing in time and converges to $u_{\nu}$ which
already finishes the proof.
\end{proof}

\subsection{Rost's reversed Root barrier}

Root's solution lets $X$ diffuse as much as possible before it stops it. 
Rost \cite{MR0346920} showed that one can also construct a closed subset $R$ of $\left[0,\infty\right]\times\left[-\infty,\infty\right]$,
the so-called \emph{reversed Root barrier} $R$ that lets $X$ diffuse as little as possible. 
More precisely, we call $R$ a reversed Root barrier if it is relatively closed in $(0,\infty) \times \mathbb{R}$ and $$\left(t,x\right)\in R \mbox{ implies }\left(s,x\right)\in R\;\forall s\leq t.$$ Reversed barriers can always be represented as $\left\{ 0 < t \leq f_R(x) \right\}$, where $f_R$ is upper-semicontinuous on $\mathbb{R}$.
%We also say that a reversed Root barrier $R$ is $\left(\mu,\nu\right)$-regular
%if $R\supset\left[0,\infty\right]\times\mathcal{N}^{\mu,\nu}$ and
%denote the set of $\left(\mu,\nu\right)$-regular reversed Root barriers
%with $\mathcal{R}_{\mu,\nu}^{\text{rev}}$. 
We now briefly show that under an additional assumption on the supports of $\mu$ and $\nu$, the methods of the previous section immediately transfer; especially, this allows to give a (constructive) PDE proof of the existence of a solution to $\operatorname{SEP}_{\sigma,\mu,\nu}$ by a reversed Root barrier.

\begin{assumption} \label{asn:supports}
There exists $V$ open such that 
\begin{equation}
\operatorname{supp}(\mu) \subset V \subset \operatorname{supp}(\nu)^c.
\label{eq:supports}
\end{equation} 
\end{assumption}

\begin{thm}
Consider the following statements:
\begin{enumerate}
\item there exists a reversed Root barrier $R$ such
that $\tau^{R}=\inf\left\{ t>0:\left(t,X_{t}\right)\in R\right\} \in\operatorname{SEP}_{\sigma,\mu,\nu}$,
\item there exists a viscosity solution $u\in C\left(\left[0,\infty\right],\left[-\infty,\infty\right]\right)$
of
\begin{equation}
\left\{ \begin{array}{rcl}
\partial_{t}u & = & \min\left(0,\frac{\sigma^{2}}{2}\Delta u\right)\text{ on }\left(0,\infty\right)\times\mathbb{R},\\
u\left(0,.\right) & = & u_{\mu}\left(.\right)-u_{\nu}\left(.\right).
\end{array}\right.\label{eq:obs problem-1}
\end{equation}

\end{enumerate}
Then under Assumption \ref{as:main assumption}, (i) $\Rightarrow$ (ii), and under Assumptions \ref{as:main assumption} and \ref{asn:supports}, (ii) $\Rightarrow$ (i).
\textup{\emph{Moreover, in either case, one can take}}\textup{ 
\begin{equation}
R=\left\{ \left(t,x\right)\in\left(0,\infty\right]\times \mathbb{R}:u\left(t,x\right)= u\left(0,x\right)\right\} \text{ and }u\left(t,x\right)=-\mathbb{E}\left[\left|X_{t\wedge\tau^{R}}-x\right|\right]-u_{\nu}(x).\label{eq:representations R and u-1}
\end{equation}
}\end{thm}
\begin{proof}
The proof of the first direction, \textbf{(i) implies (ii)},
follows exactly as in Theorem \ref{thm:supersolution} and Theorem
\ref{thm:root as pde}: first one shows that every $\tau\in\operatorname{SEP}_{\sigma,\mu,\nu}$
gives a supersolution and then shows that the potential function of
the reversed Root barrier is a solution. To do so one defines approximations
to $u$ via mollification and shows that they fulfill a perturbed
version of the PDE (\ref{eq:obs problem-1}) (at this point one use
above properties of the Rost barrier), then one concludes by stability of viscosity solutions. In this case, the PDE is linear in $\partial_{t}u$ and the
uniqueness follows already from well-known results that can be found
in the literature (e.g.~\cite{CrandallIshiiLions1992,FlemingSoner2006},
though we note that the domain is unbounded which leads to some subtleties
that are treated in \cite{DiehlFrizOberhauser}). 

To see that \textbf{(ii) implies (i) } we argue similarly as in Theorem
\ref{thm:root as pde}. Set 
\[
R:=\left\{ \left(t,x\right)\in\left(0,\infty\right)\times\mathbb{R}:u\left(t,x\right)= u\left(0,x\right)\right\} 
\]
and note that since $t\mapsto u\left(t,x\right)$ is decreasing (since
$\partial_{t}u\equiv\min\left(0,\frac{\sigma^{2}}{2}\Delta u\right)\leq0$
in viscosity sense) $R$ is indeed a reversed barrier. It remains to prove
that $\tau^{R}=\inf\left\{ t:\left(t,X_{t}\right)\in R\right\} \in\operatorname{SEP}_{\sigma,\mu,\nu}$.

\textbf{Step 1} : We claim that
\begin{equation}
\lim_{t\to \infty} u(t,x) = 0, \;\;\; \forall x \in \mathbb{R},
\label{eq:lim-u}
\end{equation}
\begin{equation}
(0,\infty) \times V \subset R^c.
\label{eq:V}
\end{equation}
The first property follows by a comparison argument : since $u(0,\cdot)$ is bounded and goes to $0$ for $x \to \infty$, for each $\epsilon >0$ it can be bounded by a function $\psi_0^\epsilon$ such that $\psi_0^\epsilon$ is bounded and concave on some interval $I$, and identically equal to$\epsilon$ outside $I$. Then $u$ is bounded from above by the solution $v$ with initial data $\psi_0^\epsilon$ to $(\partial_t - \frac{\sigma^2}{2} \partial_{xx})v = 0$ on $(0,\infty)\times I$, $v \equiv \epsilon$ outside $I$. But by ellipticity $v(t,x)$ converges to $\epsilon$ as $t \to \infty$, so that $u(\infty,x) \leq \epsilon$, which proves \eqref{eq:lim-u} since $\epsilon$ was arbitrary.

For the second claim, note that by assumption $\operatorname{supp}(\nu) \cap V = \emptyset$, so that $\partial_{xx} u(0,\cdot) = -2 \mu \leq 0$ on $V$, i.e. $u(0,\cdot)$ is concave on (any connected component of) $V$. Then one can show that $u(t,\cdot)$ is concave on $V$ for all $t$, and in fact solves $(\partial_t - \frac{\sigma^2}{2}\partial_{xx})u =0$ on $V$. Finally, by local ellipticity, comparing with the solution to the heat equation on $V$ we can deduce that for any $t$ $>$ $0$, $u(t,x) < u(0,x)$ for all $x$ in $V$, i.e. $(0,\infty) \times V \subset R^c$.

\textbf{Step 2} : As in the Root barrier case, we will approximate $R$ by barriers 
$$R_\epsilon \subset R \subset R^\epsilon$$
and letting $\nu_\epsilon$, $\nu^\epsilon$ be the distributions of $X$ at the hitting times of these, prove that
\begin{equation}
u_{\nu_\epsilon} \leq u_\nu \leq u_{\nu^\epsilon}.
\label{eq:IneqpotentialsRost}
\end{equation}
To be precise, if $f_R$ is the barrier function for $R$, the barrier functions for $R_\epsilon$ and $R^\epsilon$ are defined by
$$f_{R_\epsilon}(x) = \left( f_R(x) - \epsilon\right)_+, \;\;\;\; f_{R^\epsilon}(x) = \begin{cases} f_R(x) + \epsilon, & x\notin V,\\ 0, & x\in V. \end{cases}$$

Let us prove that $u_\nu \leq u_{\nu^\epsilon}$. We define $u^\epsilon(t,x) = -\mathbb{E}\left[\left|X_{\tau^{R^\epsilon}\wedge t}-x\right|\right] - u_{\nu^\epsilon}(x)$. By the same arguments as in (i) -> (ii), one proves that $u^\epsilon$ satisfies the same PDE as $u$, and that on $(R^\epsilon)^c$ one has $\left(\partial_t - \frac{\sigma^2}{2} \partial_{xx}\right) u^\epsilon = 0$ and $u^\epsilon(t,x) < u^\epsilon(0,x)$. Let $w = u^\epsilon - u$, it will be enough to show that $w \leq 0$ (since $w(0,\cdot) = u_\nu - u_{\nu^\epsilon}$). Since $R \subset R^\epsilon$, one has that $w$ is constant in time on $R$, and satisfies $\left(\partial_t -\frac{\sigma^2}{2} \partial_{xx}\right) w \leq 0$ on $R^c$. In particular, $w$ satisfies $\partial_t w - \left(\frac{\sigma^2}{2} \partial_{xx} w\right)_+ \leq 0$, so that by comparison $\sup w = \sup w(0,\cdot)$. Noting that both $\nu$ and $\nu^\epsilon$ do not charge $V$, $w$ is affine on (each component of) $V$, so that $\sup w(0,\cdot) = \sup_{x \notin V} w(0,x)$. Then since $w$ is continuous and goes to $0$ at infinity, one can find $x \notin V$ achieving this maximum. Then if $f_R(x) < \infty$, by definition of $R^\epsilon$, there exists $(t,x) \in R^\epsilon \setminus R$. Then one has $$w(t,x) = u^\epsilon(t,x) - u(t,x) = u^\epsilon(0,x) - u(t,x) < u^\epsilon(0,x) - u(0,x) = w(0,x),$$
a contradiction. So $f_R(x) = f_{R^\epsilon}(x)=\infty$, and by \eqref{eq:lim-u} this implies $u_\nu(x) = u_{\nu^\epsilon}(x) = u_\mu(x)$. This finishes the proof of the second inequality in \eqref{eq:IneqpotentialsRost}, and the first one is proved by similar arguments which we leave to the reader.

\textbf{Step 3} : It just remains to prove that
\begin{equation}
\lim_{\epsilon \to 0} u_{\nu_\epsilon} = \lim_{\epsilon \to 0} u_{\nu^\epsilon}.
\label{eq:LimEpsRost}
\end{equation}
This is easy, once one notices (as in \cite{chaconThesis, beiglboeck2013optimal}) that shifting the barrier is the same as shifting the starting point. Indeed, extend $R$ to $\mathbb{R} \times \mathbb{R}$ by $\tilde{R} = R \cup (-\infty, 0] \times V^c$. Then letting $\tilde{\tau}$ be the hitting time of $\tilde{R}$ by the space-time process (not necessarily started at time $0$), define for a fixed $x$ the function 
$$\psi(s,y) := - \mathbb{E}^{s,y} \left[ \left|X_{\tilde{\tau}} - x\right| \right].$$
Then one has $u_{\nu^\epsilon}(x) = \int \psi(-\epsilon,y) \mu(dy)$, $u_{\nu_\epsilon}(x) = \int \psi(-\epsilon,y) \mu(dy)$. But $\psi$ satisfies $\left( \partial_t + \frac{\sigma^2}{2} \partial_{xx}\right)\psi = 0$ outside $\tilde{R}$, so that by ellipticity it is in particular continuous on $\tilde{R}^c \supset \{0\} \times \operatorname{supp}(\mu)$. This finishes the proof of \eqref{eq:LimEpsRost}, and of the theorem.
\end{proof}
A standard application of Perron's method now implies the existence of a reversed barrier solution. 
Previous proofs of reversed barrier solutions are rather involved since they make use of a heavy potential theoretic machinery (``the filling scheme'' \cite{MR0346920,chaconThesis}; though we draw attention to the recent optimal transport approach \cite{beiglboeck2013optimal} as well as work of McConnell \cite{mcconnell1991two} that is closest in spirit to our approach, though arguably more complicated).

\begin{cor}
Let $\left(\sigma,\mu,\nu\right)$ fulfill Assumptions \ref{as:main assumption} and \ref{asn:supports}.
Then there exists a reversed Root barrier $R$ such
that $\tau^{R}=\inf\left\{ t:\left(t,X_{t}\right)\in R\right\} \in\operatorname{SEP}_{\sigma,\mu,\nu}$. 
\end{cor}

\begin{rem}
As is known since the work of Chacon \cite{chaconThesis}, a sharp condition for the existence of the reversed Root barrier is that $\nu \wedge \mu = 0$ (the case of general $\mu$, $\nu$ in convex order requires additional randomization at time $0$). Our proof of (ii) $\Rightarrow$ (i) above does not work in that case without modifications, since simple examples show that one could have $\tau_R >0$ while $\tau_{R^\epsilon} = 0$ for all $\epsilon$ (where $R^\epsilon$ is the barrier shifted by $\epsilon$ outside of the support of $\mu$). It is reasonable to hope that a modification based on approximating $\nu$ by measures fulfilling Assumption \ref{asn:supports} could give a PDE proof for existence also in that general case, but we do not pursue this here.
\end{rem}
\begin{rem} The minimizing property of the reversed barrier follows exactly as in the case of Root barriers from parabolic comparison \cite{DiehlFrizOberhauser}, so we do not discuss this any further. 
We also do not spell out the uniqueness of reversed barriers here, but we leave it for the reader to verify that (as in the case of Root barriers) the free boundary always gives the maximal version of the reversed barrier. 
Similarly we do not pursue the interpretation of $u$ as generalized value function of a stopping problem.
\end{rem}
\begin{rem}
In subsequent work, Cox--Wang \cite{2013arXiv1308.4363C} studied the reversed Root barrier in a mathematical finance context.
They use work of Chacon and Rost \cite{chaconThesis,MR0346920} that ensures existence of a reversed barrier for time-homogeneous, uniformly elliptic diffusions and then use above PDE to calculate $R$.
\end{rem}
\section{\label{sec:Application-I:-calculating}Numerics: root barriers via
barles--souganidis methods}

While it falls outside the scope of this article to study numerics
of the obstacle PDE (\ref{eq:obsPDE}) in full generality we briefly
give two applications: firstly we show that classic Barles--Souganidis
method \cite{BarlesSouganidis1991,BarlesSouganidis1990} can be easily
adapted to our setting; secondly we give some concrete examples by
implementing these schemes for rather generic embedding problems.

\subsection{$\mu$ and $\nu$ of bounded support}

We give a quick construction by adapting \cite{BarlesSouganidis1991,BarlesSouganidis1990,BarlesDaherRomano1995,FlemingSoner2006}
to our setting and implementing an explicit finite differences scheme.
On $\mathcal{O}_{T}:=\left[0,T\right]\times\left[a,b\right]$ and
setting $h:=\left(\Delta t,\Delta x\right)=\left(\frac{T}{N_{T}},\frac{b-a}{N_{x}}\right)$
for $N_{T},N_{x}\in\mathbb{N}$ large enough we define the time-space
mesh of points 
\[
\mathcal{G}_{h}:=\left\{ t_{n}:t_{n}=n\Delta t,\, n=0,1,\cdots,N_{T}\right\} \times\left\{ x_{j}:x_{j}=a+j\Delta x,\, j=0,1,\cdots,N_{x}\right\} .
\]
Let $\mathcal{B}\left(\mathcal{O}_{T},\mathbb{R}\right)$ be the set
of bounded functions from $\mathcal{O}_{T}$ to $\mathbb{R}$ and
$\mathcal{BUC}\left(\mathcal{O}_{T},\mathbb{R}\right)\subset\mathcal{B}\left(\mathcal{O}_{T},\mathbb{R}\right)$
the subset of bounded uniformly continuous functions. Take $\psi\in\mathcal{BUC}\left(\mathcal{O}_{T},\mathbb{R}\right)$,
we define its projection on $\mathcal{G}_{h}$ by $\psi^{h}:\mathcal{O}_{T}^{h}\to\mathbb{R}$
with $\mathcal{O}_{T}^{h}:=\left[0,T+\Delta t\right)\times\left[a-\Delta x/2,b+\Delta x/2\right]$
as $\psi^{h}\left(t,x\right):=\psi\left(t_{n},x_{j}\right)$ when
$(t,x)\in[t_{n},t_{n+1})\times[x_{j}-\Delta x/2,x_{j}+\Delta x/2)$
for some $n\in\left\{ 0,1,\cdots,N_{T}\right\} $ and $j\in\left\{ 0,1,\cdots,N_{x}\right\} $;
of course $\psi^{h}\in\mathcal{B}\left(\mathcal{O}_{T},\mathbb{R}\right)$.
Denote the approximation to the solution $u\in\mathcal{BUC}\left(\mathcal{O}_{T},\mathbb{R}\right)$
of (\ref{eq:obsPDE}) by $u^{h}\in\mathcal{B}\left(\mathcal{O}_{T},\mathbb{R}\right)$.
Define the operator $S^{h}:\mathcal{B}\left(\mathcal{O}_{T},\mathbb{R}\right)\times\left[0,T\right]\times\left[a,b\right]\mapsto\mathbb{R}$
as 
\[
S^{h}\left[u^{h}\right]\left(t,x\right):=\left\{ \begin{array}{ll}
u_{\mu}\left(x\right) & ,\left(t,x\right)\in\left[0,\Delta t\right)\times\left(a,b\right)\\
u^{h}\left(t,x\right)+\frac{\Delta t\,\sigma^{h}\left(t,x\right)}{2\left(\Delta x\right){}^{2}}\left(u^{h}\left(t,x+\Delta x\right)-2u^{h}\left(t,x\right)+u^{h}\left(t,x-\Delta x\right)\right) & ,\left(t,x\right)\in\left[\Delta t,T\right]\times\left(a,b\right)\\
u_{\mu}\left(x\right) & ,\left(t,x\right)\in\left[0,T\right]\times\left\{ a,b\right\} 
\end{array}\right.
\]
where we assume that the usual CFL condition $\operatorname{CFL}:=\Delta t\left|\sigma\right|_{\infty;\left[a,b\right]\times\left[0,T\right]}<\left(\Delta x\right){}^{2}$
holds. The values of $u^{h}$ are computed by solving for $u^{h}\left(t,x\right)$
in $G(.)=0$ where $G:\left(0,\infty\right){}^{2}\times\mathcal{O}_{T}\times\mathbb{R}\times\mathcal{B}\left(\mathcal{O}_{T},\mathbb{R}\right)\mapsto\mathbb{R}$
is defined as 
\[
G\left(h,\left(t+\Delta t,x\right),u^{h}\left(t+\Delta t,x\right),u^{h}\right):=\min\left\{ u^{h}\left(t+\Delta t,x\right)-u_{\nu}\left(x\right),u^{h}\left(t+\Delta t,x\right)-S^{h}\left[u^{h}\right]\left(t,x\right)\right\} 
\]
By \cite{BarlesSouganidis1991,BarlesSouganidis1990} we only have
to guarantee that the operator $S^{h}\left[.\right]\left(.\right)$
and the PDE (\ref{eq:obsPDE}) satisfies along some sequence $h:=\left(\Delta t,\Delta x\right)$
converging to $\left(0,0\right)$ the following properties:
\begin{itemize}
\item \emph{Monotonicity}. $G\left(h,\left(t,x\right),r,f^{h}\right)\leq G\left(h,\left(t,x\right),r,g^{h}\right)$
whenever $f\leq g$ with $f,g\in\mathcal{B}$ (and for finite values
of $h,t,x,r$);
\item \emph{Stability}. For every $h>0$, the scheme has a solution $u^{h}$
on $\mathcal{G}_{h}$ that is uniformly bounded independently of $h$
(under the CFL condition, see above);
\item \emph{Consistency}. For any $\psi\in C_{b}^{\infty}\left(\mathcal{O}_{T};\mathbb{R}\right)$
and $\left(t,x\right)\in\mathcal{O}_{T}$, we have (under the $\operatorname{CFL}$
condition, see above): 
\begin{eqnarray*}
 &  & \lim_{\left(h,\xi,t_{n}+\Delta t,x_{j}\right)\rightarrow\left(0,0,t,x\right)}\left(\left(\psi\left(t_{n}+\Delta t,x_{j}\right)+\xi\right)-u_{\nu}(x)\right)\wedge\frac{\left(\psi\left(t_{n}+\Delta t,x_{j}\right)+\xi\right)-S^{h}\left[\psi^{h}+\xi\right]\left(t_{n},x_{j}\right)}{\Delta t}\\
 & = & \min\left\{ \psi\left(t,x\right)-u_{\nu}\left(x\right),\left(\partial_{t}\psi-\frac{\sigma^{2}}{2}\Delta\psi\right)\left(t,x\right)\right\} 
\end{eqnarray*}

\item \emph{Strong uniqueness}. if the locally bounded USC {[}resp.~LSC{]}
function $u$ {[}resp.~$v${]} is a viscosity subsolution {[}resp.
supersolution{]} of (\ref{eq:obsPDE}) then $u\leq v$ in $\mathcal{O}_{T}$;\end{itemize}
\begin{prop}
Let $T\in\left(0,\infty\right)$. Assume $\mu,\nu$ have compact support
and $\left(\sigma,\mu,\nu\right)$ fulfill Assumption \ref{as:main assumption}.
Then $u^{h}\in\mathcal{B}\left(\left[0,T\right]\times\mathbb{R},\mathbb{R}\right)$
and 
\[
\left|u^{h}-u\right|_{\infty;\left[0,T\right]\times\mathbb{R}}\rightarrow0\text{ as }h\rightarrow\left(0,0\right)
\]
where $u$ denotes the unique viscosity solution of linear growth
of (\ref{eq:obsPDE}) on $\left[0,T\right]$.\end{prop}
\begin{proof}
This follows by verification of the assumptions in \cite[Theorem 2.1]{BarlesSouganidis1991,BarlesDaherRomano1995}:
strong uniqueness comes from our comparison theorem, existence from
Corollary \ref{cor:existence root}. Monotonicity, stability and consistency
follow by a direct calculation which we do not spell out here. The
rest of the proof is given by following closely \cite[Theorem 2.1]{BarlesSouganidis1991,BarlesDaherRomano1995}
combined with the remarks on the first example in \cite[Section 5]{BarlesDaherRomano1995}:
one first shows that the operator $S^{h}\left[.\right]\left(.\right)$
approximates the diffusion component of (\ref{eq:obsPDE}) and subsequently
adds the barrier to recover the full equation (\ref{eq:obsPDE}).
One finally concludes as in \cite[p130]{BarlesDaherRomano1995} by
semi-relaxed limits in combination with our comparison result, Theorem
(\ref{thm:compforOSB}).
\end{proof}
\begin{figure}
\begin{minipage}[t]{0.45\textwidth}%
\begin{center}
\includegraphics[bb=0bp 0bp 300bp 194bp,clip,scale=0.6]{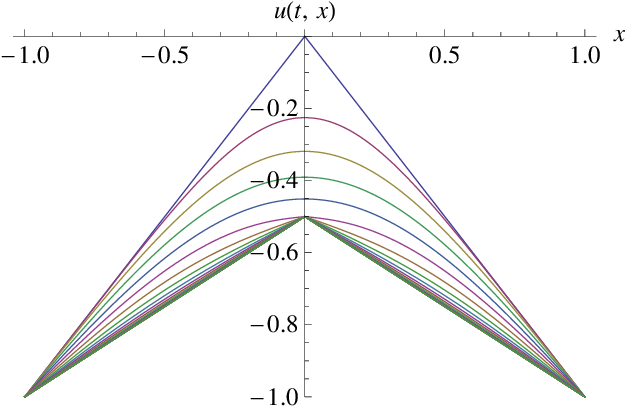}
\par\end{center}%
\end{minipage}\hfill{}%
\begin{minipage}[t]{0.45\textwidth}%
\begin{center}
\includegraphics[clip,scale=0.6]{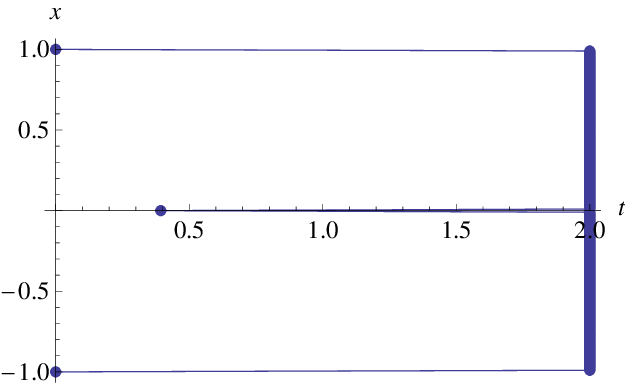}
\par\end{center}%
\end{minipage}\caption{$\sigma=1$, $\mu=\delta_{0}$ and $\nu=\frac{1}{4}\delta_{-1}+\frac{1}{2}\delta_{0}+\frac{1}{4}\delta_{1}$.
Above finite difference scheme is used with $\operatorname{CFL}=0.2$
and $50.10^{3}$ time steps on the time domain $\left[0,2\right]$
and spatial domain $\left[-1,1\right]$. The left plot shows that
for $t_{0}\sim0.39$ the potentials touch at $x=0$ which determines
the spike of the Root barrier depicted in the right plot.}
\end{figure}
\begin{figure}
\begin{minipage}[t]{0.45\textwidth}%
\begin{center}
\includegraphics[bb=0bp 0bp 300bp 194bp,clip,scale=0.6]{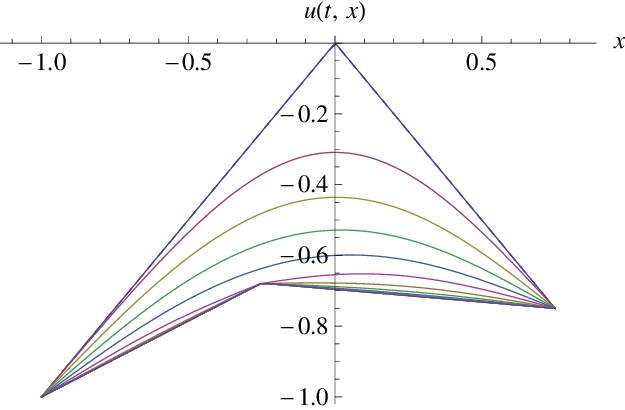}
\par\end{center}%
\end{minipage}\hfill{}%
\begin{minipage}[t]{0.45\textwidth}%
\begin{center}
\includegraphics[clip,scale=0.6]{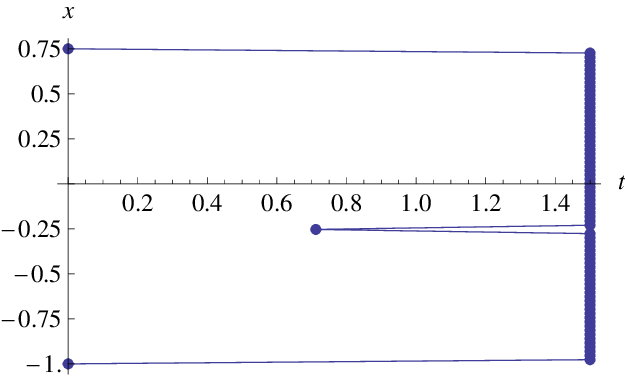}
\par\end{center}%
\end{minipage}\caption{$\sigma=1$,$\mu=\delta_{0}$ and $\nu=\frac{2}{7}\delta_{-1}+\frac{1}{4}\delta_{-\frac{1}{4}}+\frac{13}{28}\delta_{\frac{3}{4}}$. }
\end{figure}
\begin{figure}
\begin{minipage}[t]{0.45\textwidth}%
\begin{center}
\includegraphics[bb=0bp 0bp 300bp 194bp,clip,scale=0.6]{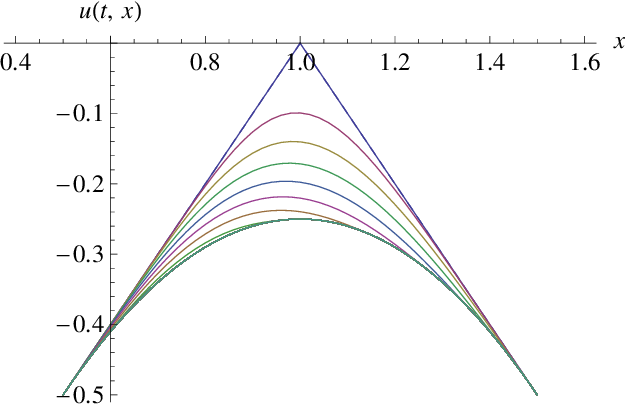}
\par\end{center}%
\end{minipage}\hfill{}%
\begin{minipage}[t]{0.45\textwidth}%
\begin{center}
\includegraphics[clip,scale=0.6]{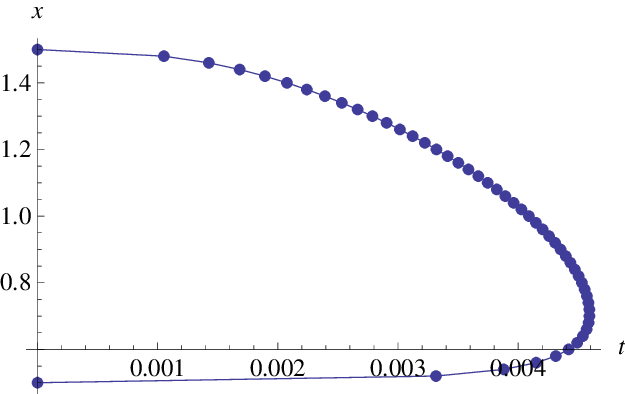}
\par\end{center}%
\end{minipage}\caption{$\sigma\left(x\right)=x$, $\mu=\delta_{1}$ and $\nu=\mathcal{U}\left(\left[\frac{1}{2},\frac{3}{2}\right]\right)$.
Initial and target measure are symmetric but the barrier is asymmetric
due to $\sigma\left(x\right)=x$.}
\end{figure}
\begin{figure}
\begin{minipage}[t]{0.45\textwidth}%
\begin{center}
\includegraphics[bb=0bp 0bp 300bp 194bp,clip,scale=0.6]{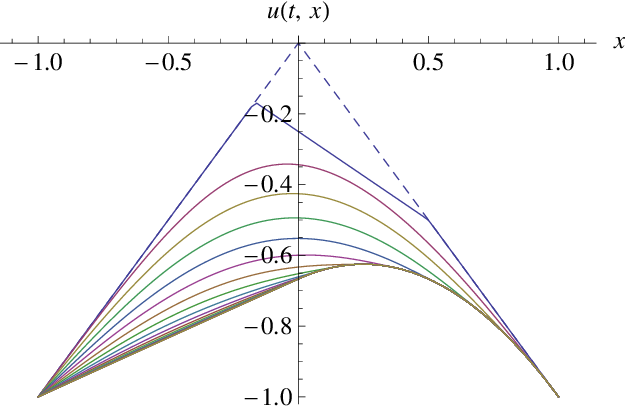}
\par\end{center}%
\end{minipage}\hfill{}%
\begin{minipage}[t]{0.45\textwidth}%
\begin{center}
\includegraphics[clip,scale=0.6]{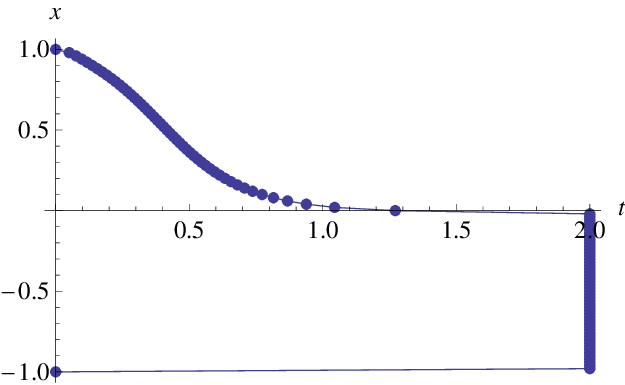}
\par\end{center}%
\end{minipage}\caption{$\sigma\left(x\right)=1$,$\mu=\frac{3}{4}\delta_{-\frac{1}{6}}+\frac{1}{4}\delta_{0.5}$
and $\nu=\frac{1}{3}\delta_{-1}+\frac{2}{3}\mathcal{U}\left(\left[0,1\right]\right).$ }
\end{figure}

\subsection{$\mu$ and $\nu$ of unbounded support}

For simplicity we restrict ourselves to embeddings into Brownian motion
(i.e.~$\sigma\equiv1$). In this case recent results of Jakobsen
\cite{jakobsen2003rate} apply and give a convergence rate of order
$\frac{1}{2}$. Denote $h=\left(\Delta t,\Delta x\right)$ and consider
schemes of the type 
\[
u^{h}\left(t+\Delta t,x\right)=\max\left\{ u_{\nu}\left(x\right),S_{\Delta t}u^{h}\left(t,x\right)\right\} 
\]
where $S_{\Delta t}$ is the (formal) solution operator associated
to the heat equation $\partial_{t}w-\frac{1}{2}\Delta w=0$. In the
case that we use a finite difference method this scheme can be written
as

\begin{equation}
\min\left\{ u^{h}\left(t+\Delta t,x\right)-u_{\nu}\left(x\right),\frac{u^{h}\left(t+\Delta t,x\right)-u^{h}\left(t,x\right)}{\Delta t}-\frac{u^{h}\left(t,x+\Delta x\right)-2u^{h}\left(t,x\right)+u^{h}\left(t,x-\Delta x\right)}{2\left(\Delta x\right)^{2}}\right\} =0.\label{eq:scheme}
\end{equation}
A direct calculation also shows that this is equivalent to (see Jakobsen
\cite[page 11 in Section 3]{jakobsen2003rate})

\[
u^{h}\left(t+\Delta t,x\right)=\max\left\{ u_{\nu}\left(x\right),u^{h}\left(t,x\right)+\frac{\Delta t}{2\left(\Delta x\right){}^{2}}\left(u^{h}\left(t,x+\Delta x\right)-2u^{h}\left(t,x\right)+u^{h}\left(t,x-\Delta x\right)\right)\right\} 
\]
and above representation is advantageous for the proof.
\begin{prop}[{\cite[Section 3]{jakobsen2003rate}}]
Let $\left(1,\mu,\nu\right)$ fulfill Assumption (\ref{as:main assumption}).
Then there exists a unique $u^{h}$ solving (\ref{eq:scheme}). Further,
if $\Delta t\leq\left(\Delta x\right)^{2}$ and $u_{0}^{h}$ is an
approximation of $u_{0}$ which is bounded independently of $h$ then
\[
\left|u-u^{h}\right|_{\infty}\lesssim\sup_{\left[0,\Delta t\right)\times\mathbb{R}}\left|u-u_{0}^{h}\right|+\left(\Delta x\right)^{1/2}
\]
\end{prop}
\begin{proof}
This is a direct consequence of \cite[Section 3]{jakobsen2003rate}
which shows that one can replace the Barles--Souganidis assumptions
by more special conditions ($C1-C5$ in \cite[Section 2]{jakobsen2003rate}).
A direct calculation shows then that these are fulfilled for the finite-difference
scheme under our assumptions.
\end{proof}

\section{a comparison for obstacle pdes and a lemma about jets}

Comparison theorems for obstacle problems can be found in the literature,
see \cite{jakobsen2002continuous,jakobsen2003rate,el1997reflected}.
However, due to the unboundedness of the coefficients as well as space
they do not cover our setup. We provide a complete proof by revisiting
work of \cite{jakobsen2002continuous,jakobsen2003rate,DiehlFrizOberhauser}.
It also establishes Hölder regularity in space of viscosity solutions.

\subsection{A Comparison Theorem for the obstacle problem}
\begin{thm}
\label{thm:compforOSB}Let $h\in C\left(\mathbb{R},\mathbb{R}\right)$
be of linear growth, i.e.~ $\exists c>0$ such that 
\[
\left|h\left(x\right)\right|\leq c\left(1+\left|x\right|\right)\text{ for all }x\in\mathbb{R}
\]
and $\sigma\in C\left(\left[0,T\right]\times\mathbb{R},\mathbb{R}\right)$
Lipschitz in space, uniformly in time ($\sup_{t}\left|\sigma\left(t,.\right)\right|_{Lip}<\infty$).
Define  
\[
F_{obs}\left(t,x,r,a,p,M\right)=\min\left(r-h\left(x\right),a-\frac{\sigma^{2}\left(t,x\right)}{2}m\right).
\]
Let $u\in USC\left(\left[0,T\right]\times\mathbb{R},\mathbb{R}\right)$
be a viscosity subsolution and $v\in LSC\left(\left[0,T\right]\times\mathbb{R},\mathbb{R}\right)$
a viscosity supersolution of the PDE
\[
F_{obs}\left(t,x,u,\partial_{t}u,Du,D^{2}u\right)\leq0\leq F_{obs}\left(t,x,v,\partial_{t}v,Dv,D^{2}v\right)\text{ on }\left(0,T\right)\times\mathbb{R}
\]
Further assume that $\forall\left(t,x\right)\in\left[0,T\right)\times\mathbb{R}$,
$u\left(t,x\right),-v\left(t,x\right)\leq C\left(1+\left|x\right|\right)$
for some constant $C>0$ and that $u\left(0,.\right)$ and $v\left(0,.\right)$
are $\delta$-Hölder continuous. Then there exists a constant $c\geq0$
s.t.~$\forall\left(t,x,y\right)\in\left[0,T\right)\times\mathbb{R}\times\mathbb{R}$
\[
u\left(t,x\right)-v\left(t,y\right)\leq\sup_{z\in\mathbb{R}}\left(u_{0}\left(z\right)-v_{0}\left(z\right)\right)+c\inf_{\alpha>0}\left\{ \alpha^{-\frac{1}{2-\delta}}+\alpha\left|x-y\right|^{2}\right\} .
\]
Direct consequences of this estimate are
\begin{enumerate}
\item \label{enu:comp}$u_{0}\leq v_{0}$ implies $u\leq v$ on $\left[0,T\right)\times\mathbb{R}$,
\item if $u$ is also a supersolution (viz.~$u$ is a viscosity solution)
then $u$ is $\delta$-Hölder continuous in space uniformly in time
on $\left[0,T\right)$, i.e.
\[
\sup_{t\in\left[0,T\right)}\left|u\left(t,.\right)\right|_{C^{\delta}\left(\mathbb{R}\right)}<\infty.
\]

\end{enumerate}
\end{thm}
\begin{proof}
Wlog we can replace the parabolic part in $F$ with $\partial_{t}w-\sigma^{2}\Delta w-w$
(by replacing $u$ resp.~$v$ with $e^{-t}u$ resp.~$e^{-t}v$).
Further we can assume that $\forall\overline{\epsilon}>0$, $u$ is
a subsolution of 
\begin{eqnarray}
F_{obs}\left(t,x,u,\partial_{t}u,Du,D^{2}u\right) & \leq & -\frac{\overline{\epsilon}}{\left(T-t\right)^{2}}\label{eq:t infty}\\
\lim_{t\uparrow T}u\left(t,x\right) & = & -\infty\text{ uniformly on }\mathcal{O}\nonumber 
\end{eqnarray}
(by replacing $u$ with $u-\frac{\overline{\epsilon}}{T-t}$). Define
for $\alpha>0$,$\epsilon>0$ 
\[
\psi\left(t,x,y\right)=u-v-\phi\left(t,x,y\right)\text{ with }\phi\left(t,x,y\right)=e^{\lambda t}\alpha\left|x-y\right|^{2}+\epsilon\left(\left|x\right|^{2}+\left|y\right|^{2}\right)
\]
and 
\[
m_{\alpha,\epsilon}^{0}=\sup_{\mathbb{R}\times\mathbb{R}}\psi\left(0,x,y\right)^{+}\text{ and }m_{\alpha,\epsilon}=\sup_{\left[0,T\right]\times\mathbb{R}\times\mathbb{R}}\psi\left(t,x,y\right)-m_{\alpha,\epsilon}^{0}.
\]
The growth assumptions on $u$ and $v$ together with (\ref{eq:t infty})
guarantee for every $\alpha>0$, $\epsilon>0$ the existence of a
triple $\left(\hat{t},\hat{x},\hat{y}\right)\in\left[0,T\right)\times\mathbb{R}\times\mathbb{R}$
s.t.~
\[
m_{\alpha,\epsilon}+m_{\alpha,\epsilon}^{0}=\psi\left(\hat{t},\hat{x},\hat{y}\right).
\]
The proof strategy is classic: the above implies that $\forall\alpha>0,\epsilon>0$
and $\forall\left(t,x,y\right)$ 
\begin{equation}
u\left(t,x\right)-v\left(t,y\right)\leq m_{\alpha,\epsilon}+m_{\alpha,\epsilon}^{0}+e^{\lambda t}\alpha\left|x-y\right|^{2}+\epsilon\left(\left|x\right|^{2}+\left|y\right|^{2}\right).\label{eq:u-v}
\end{equation}
Using the Hölder continuity of $u_{0}$ and $v_{0}$ we immediately
get an upper bound for $m_{\alpha,\epsilon}^{0}$ 
\begin{eqnarray*}
m_{\alpha,\epsilon}^{0} & \leq & u_{0}\left(\hat{x}\right)-v_{0}\left(\hat{x}\right)+\left|v_{0}\right|_{\delta}\left|\hat{x}-\hat{y}\right|^{\delta}-\alpha\left|\hat{x}-\hat{y}\right|^{2}\\
 & \leq & \left|u_{0}-v_{0}\right|+c\alpha^{-\frac{1}{2-\delta}}
\end{eqnarray*}
and below we use the parabolic theorem of sums to show that 
\begin{equation}
m_{\alpha,\epsilon}\leq C\alpha^{-\frac{1}{2-\delta}}+k\epsilon+\omega_{a}\left(\epsilon\right)\label{eq:m_a_e}
\end{equation}
where $\omega_{\alpha}\left(.\right)$ is a modulus of continuity
for every $\alpha>0$. Plugging these two estimates into (\ref{eq:u-v})
gives 
\begin{eqnarray*}
u\left(t,x\right)-v\left(t,y\right) & \leq & \left|u_{0}-v_{0}\right|+\left(c+C\right)\alpha^{-\frac{1}{2-\delta}}+e^{\lambda t}\alpha\left|x-y\right|^{2}+\epsilon\left(\left|x\right|^{2}+\left|y\right|^{2}\right).
\end{eqnarray*}
Now letting $\epsilon\rightarrow0$ and subsequently optimizing over
$\alpha$ yields the key estimate 
\begin{eqnarray*}
u\left(t,x\right)-v\left(t,y\right) & \leq & \left|u_{0}-v_{0}\right|+\inf_{\alpha>0}\left(\left(c+C\right)\alpha^{\frac{-1}{2-\delta}}+e^{\lambda t}\alpha\left|x-y\right|^{2}\right).
\end{eqnarray*}
Applying it with $x=y$ gives point (\ref{enu:comp}) of our statement.
Applying it with a viscosity solution $u=v$ gives 
\[
u\left(t,x\right)-u\left(t,y\right)\leq\inf_{\alpha>0}\left(\left(c+C\right)\alpha^{\frac{-1}{2-\delta}}+e^{\lambda t}\alpha\left|x-y\right|^{2}\right)
\]
and the estimate $\inf_{\alpha>0}\left\{ \alpha^{-\frac{1}{2-\delta}}+\alpha r^{2}\right\} =cr^{\delta}$
yields the $\delta$-Hölder regularity.

It remains to show (\ref{eq:m_a_e}). Below we assume $m_{\alpha,\epsilon}\geq0$
and derive the upper bound (\ref{eq:m_a_e}) (which then also holds
if $m_{\alpha,\epsilon}<0$). Note that $m_{\alpha,\epsilon}\geq0$
implies $\hat{t}>0$. The parabolic Theorem of sums \cite[Theorem 8.3]{MR1118699UserGuide}
shows existence of 
\[
\left(a,D_{x}\psi\left(\hat{t},\hat{x},\hat{y}\right),X\right)\in\mathcal{\overline{P}}_{\mathcal{O}}^{2,+}u\left(\hat{t},\hat{x}\right)\text{ and }\left(b,D_{y}\psi\left(\hat{t},\hat{x},\hat{y}\right),Y\right)\in\overline{\mathcal{P}}_{\mathcal{O}}^{2,-}v\left(\hat{t},\hat{x}\right)
\]
such that 
\begin{equation}
a-b=\dot{\psi}\left(\hat{t},\hat{x},\hat{y}\right)\text{ and }\left(\begin{array}{cc}
X & 0\\
0 & -Y
\end{array}\right)\leq ke^{\lambda t}\alpha\left(\begin{array}{cc}
1 & -1\\
-1 & 1
\end{array}\right)+k\epsilon\left(\begin{array}{cc}
1 & 0\\
0 & 1
\end{array}\right).\label{eq:TOS}
\end{equation}
Since $u$ is a subsolution resp\@. $v$ is a supersolution 
\begin{eqnarray*}
\min\left(a-\frac{\sigma^{2}\left(\hat{t},\hat{x}\right)}{2}X-u\left(\hat{t},\hat{x}\right),u\left(\hat{t},\hat{x}\right)-h\left(\hat{x}\right)\right) & \leq & 0\\
\min\left(b-\frac{\sigma^{2}\left(\hat{t},\hat{y}\right)}{2}Y-v\left(\hat{t},\hat{x}\right),v\left(\hat{t},\hat{x}\right)-h\left(\hat{x}\right)\right) & \geq & 0
\end{eqnarray*}
and subtracting the second inequality from the first leads to 
\[
\min\left(a-b-\frac{\sigma^{2}\left(t,x\right)}{2}X+\frac{\sigma\left(\hat{t},\hat{y}\right)}{2}Y-u\left(\hat{t},\hat{x}\right)+v\left(\hat{t},\hat{y}\right),u\left(\hat{t},\hat{x}\right)-v\left(\hat{t},\hat{y}\right)-h\left(\hat{x}\right)+h\left(\hat{y}\right)\right)\leq0.
\]
First assume the second term in the $\min$ is less than or equal
to $0$. This gives 
\[
u\left(\hat{t},\hat{x}\right)-v\left(\hat{t},\hat{y}\right)\leq h\left(\hat{x}\right)-h\left(\hat{y}\right)\leq\left|h\right|_{\delta}\left|\hat{x}-\hat{y}\right|^{\delta}
\]
hence we get the estimate 
\begin{equation}
m_{\alpha,\epsilon}\leq\left|h\right|_{\delta}\left|\hat{x}-\hat{y}\right|^{\delta}.\label{eq:case1}
\end{equation}
Now assume the first term in the $\min$ is less than or equal to
$0$. This gives
\[
\dot{\psi}\left(\hat{t},\hat{x},\hat{y}\right)+u\left(\hat{t},\hat{x}\right)-v\left(\hat{t},\hat{y}\right)\leq\frac{\sigma^{2}\left(t,x\right)}{2}X-\frac{\sigma\left(\hat{t},\hat{y}\right)}{2}Y
\]
hence from the definition of $\psi$ resp.~$m_{\alpha,\epsilon}$
it then follows that 
\[
\lambda e^{\lambda t}\alpha\left|\hat{x}-\hat{y}\right|^{2}+m_{\alpha,\epsilon}\leq\frac{1}{2}\left(\sigma^{2}\left(\hat{t},\hat{x}\right)X-\sigma^{2}\left(\hat{t},\hat{y}\right)Y\right).
\]
Estimate the rhs by multiplying the matrix inequality (\ref{eq:TOS})
from the left respectively right with the vector $\left(\sigma\left(\hat{t},\hat{x}\right),\sigma\left(\hat{t},\hat{y}\right)\right)$
resp. $\left(\sigma\left(\hat{t},\hat{x}\right),\sigma\left(\hat{t},\hat{y}\right)\right)^{t}$
to get 
\begin{eqnarray}
\lambda e^{\lambda t}\alpha\left|\hat{x}-\hat{y}\right|^{2}+m_{\alpha,\epsilon} & \leq & ke^{\lambda t}\alpha\left|\sigma\left(\hat{t},\hat{x}\right)-\sigma\left(\hat{t},\hat{y}\right)\right|^{2}+k\epsilon\left(\sigma^{2}\left(\hat{t},\hat{x}\right)+\sigma^{2}\left(\hat{t},\hat{y}\right)\right)\nonumber \\
 & \leq & ke^{\lambda t}\alpha\left|\sigma\right|_{1}\left|\hat{x}-\hat{y}\right|^{2}+k\epsilon\left(\sigma^{2}\left(\hat{t},\hat{x}\right)+\sigma^{2}\left(\hat{t},\hat{y}\right)\right).\label{eq:case2}
\end{eqnarray}
By adding (\ref{eq:case1}) and (\ref{eq:case2}) together and choosing
$\lambda=\left|\sigma\right|_{1}k+1$ we finally get 
\[
m_{\alpha,\epsilon}\leq\left|h\right|_{\delta}\left|\hat{x}-\hat{y}\right|^{\delta}-e^{\lambda t}\alpha\left|\hat{x}-\hat{y}\right|^{2}+k\epsilon\left(\sigma^{2}\left(\hat{t},\hat{x}\right)+\sigma^{2}\left(\hat{t},\hat{y}\right)\right).
\]
We estimate the sum of the first two terms on the rhs by using that
$\sup_{r\geq0}\left(r^{\delta}+\frac{\alpha}{2}r^{2}\right)\leq c\alpha^{-\frac{1}{2-\delta}}$
and the last term using linear growth of $\sigma^{2}$ to arrive at
\[
m_{\alpha,\epsilon}\leq C\alpha^{-\frac{1}{2-\delta}}+k\epsilon\left(1+\left|\hat{x}\right|^{2}+\left|\hat{y}\right|^{2}\right).
\]
By lemma (\ref{lem:space_penalty}) we can replace $\epsilon\left(\left|\hat{x}\right|^{2}+\left|\hat{y}\right|^{2}\right)$
by a modulus $\omega_{\alpha}\left(\epsilon\right)$ (i.e.~for every
$\alpha>0$, $\omega_{\alpha}\in C\left(\left[0,\infty\right),\left[0,\infty\right)\right)$,
$\omega_{\alpha}\left(0\right)=0$ and $\omega_{\alpha}$ is non-decreasing),
i.e.
\[
m_{\alpha,\epsilon}\leq C\alpha^{-\frac{1}{2-\delta}}+k\epsilon+\omega_{a}\left(\epsilon\right).
\]
Hence we have shown that $\forall\alpha>0$ 
\[
\limsup_{\epsilon}m_{\alpha,\epsilon}=C\alpha^{-\frac{1}{2-\delta}},
\]
and
\[
u\left(t,x\right)-v\left(t,y\right)\leq m_{\alpha,\epsilon}+m_{\alpha,\epsilon}^{0}+e^{\lambda t}\frac{\alpha}{2}\left|\hat{x}-\hat{y}\right|^{2}+\epsilon\left(\left|\hat{x}\right|^{2}+\left|\hat{y}\right|^{2}\right).
\]
\end{proof}
\begin{lem}
\label{lem:space_penalty}Let $f\in USC\left(\left[0,T\right]\times\mathbb{R}\times\mathbb{R}\right)$
and bounded from above. Set 
\begin{eqnarray*}
m & := & \sup_{\left[0,T\right]\times\mathbb{R}\times\mathbb{R}}f\left(t,x,y\right)\\
m_{\epsilon} & := & \sup_{\left[0,T\right]\times\mathbb{R}\times\mathbb{R}}f\left(t,x,y\right)-\epsilon\left(\left|x\right|^{2}+\left|y\right|^{2}\right)\text{ for }\epsilon>0.
\end{eqnarray*}
Denote with $\left(\hat{t}_{\epsilon},\hat{x}_{\epsilon},\hat{y}_{\epsilon}\right)$
points where the $sup$ is attained. Then
\begin{enumerate}
\item $\lim_{\epsilon\rightarrow0}m_{\epsilon}=\sup_{\left[0,T\right]\times\mathbb{R}\times\mathbb{R}}f\left(t,x,y\right)$,
\item $\epsilon\left(\left|\hat{x}_{\epsilon}\right|^{2}+\left|\hat{y}_{\epsilon}\right|^{2}\right)\rightarrow_{\epsilon}0$.
\end{enumerate}
\end{lem}
\begin{proof}
By definition of a supremum there exists for every $\eta>0$ a triple
$\left(t_{\eta},x_{\eta},y_{\eta}\right)\in\left[0,T\right]\times\mathbb{R}\times\mathbb{R}$
such that $f\left(t_{\eta},x_{\eta},y_{\eta}\right)>m-\eta$. Fix
$\eta>0$ and take $\epsilon'$ small enough s.t.~$\epsilon'\left(\left|x_{\eta}\right|^{2}+\left|y_{\eta}\right|^{2}\right)\leq\eta$.
Then $\forall\epsilon\in\left[0,\epsilon'\right]$ we have 
\[
m\geq m_{\epsilon}\geq f\left(t_{\eta},x_{\eta},y_{\eta}\right)-\epsilon'\left(\left|x_{\eta}\right|^{2}+\left|y_{\eta}\right|^{2}\right)\geq f\left(t_{\eta},x_{\eta},y_{\eta}\right)-\eta\geq m-2\eta.
\]
Since $\eta$ can be arbitrary small and $\epsilon\mapsto m_{\epsilon}$
is non-increasing the first claim follows. From the above estimate
and the boundedness of $f$ from above also show that
\[
k_{\epsilon}=\epsilon\left(\left|\hat{x}_{\epsilon}\right|^{2}+\left|\hat{y}_{\epsilon}\right|^{2}\right)
\]
is bounded. Hence there exists a subsequence of $\left(k_{\epsilon}\right)_{\epsilon>0}$
which we denote with slight abuse of notation again as $\left(k_{\epsilon}\right)_{\epsilon>0}$
which converges to some limit denoted $k$($\geq0$). Now 
\[
f\left(\hat{t}_{\epsilon},\hat{x}_{\epsilon},\hat{y}_{\epsilon}\right)-k_{\epsilon}\leq m-k_{\epsilon}
\]
and from the first part we can send $\epsilon$ to $0$ along the
subsequence and see that $m-k\leq m$, hence $k=0$. Since we have
shown that every subsequence $\left(k_{\epsilon}\right)$ converges
to $0$ the second statement follows. 
\end{proof}

\subsection{A Lemma about sup- and superjets}

We now provide the proof of the Lemma that plays a crucial role in
the proof of Theorem (\ref{thm:root as pde}). It describes the elements
in the sub and superjets $\mathcal{P}_{\mathcal{O}}^{2,-}\left(u\right)$
and $\mathcal{P}_{\mathcal{O}}^{2,+}\left(u\right)$ for functions
which are only left- and right-differentiable.
\begin{lem}
\label{lem:jets for u-1}Let $v\in C\left(\left(0,\infty\right)\times\mathbb{R},\mathbb{R}\right)$
and assume that $\forall\left(t,x\right)\mathbb{\in}\left(0,\infty\right)\times\mathbb{R}$,
$v$ has a left- and right-derivative, i.e.~the following limits
exist
\[
\partial_{t+}v\left(t,x\right)=\lim_{\epsilon\searrow0}\frac{v\left(t+\epsilon,x\right)-v\left(t,x\right)}{\epsilon}\text{ and }\partial_{t-}v\left(t,x\right)=\lim_{\epsilon\nearrow0}\frac{v\left(t+\epsilon,x\right)-v\left(t,x\right)}{\epsilon},
\]
If $\partial_{t-}v\left(t,x\right)\leq\partial_{t+}v\left(t,x\right)$
then 
\[
a\in\left[\partial_{t-}v\left(t,x\right),\partial_{t+}v\left(t,x\right)\right]\,\,\,\forall\left(a,p,m\right)\in\mathcal{P}_{\mathcal{O}}^{2,-}v\left(t,x\right).
\]
If $\partial_{t-}v\left(t,x\right)<\partial_{t+}v\left(t,x\right)$
then $\mathcal{P}_{\mathcal{O}}^{2,+}v\left(t,x\right)=\emptyset$
and if $\partial_{t-}v\left(t,x\right)=\partial_{t+}v\left(t,x\right)$
then $\forall\left(a,p,m\right)\in\mathcal{P}_{\mathcal{O}}^{2,+}v\left(t,x\right)$,
$a=\partial_{t}v\left(t,x\right)$ ($=\partial_{t-}v\left(t,x\right)=\partial_{t+}v\left(t,x\right)$).
In all the above cases, if $v$ is additionally twice continuously
differentiable in space then 
\begin{eqnarray*}
m & \leq & \Delta v\left(t,x\right)\,\,\,\forall\left(a,p,m\right)\in\mathcal{P}_{\mathcal{O}}^{2,-}v\left(t,x\right),\\
m & \geq & \Delta v\left(t,x\right)\,\,\,\forall\left(a,p,m\right)\in\mathcal{P}_{\mathcal{O}}^{2,+}v\left(t,x\right).
\end{eqnarray*}
\end{lem}
\begin{proof}
Every element $\left(a,p,m\right)\in\mathcal{P}_{\mathcal{O}}^{2,-}v\left(t,x\right)$
fulfills 
\[
v\left(t+\epsilon,x\right)-v\left(t,x\right)\geq a\epsilon+o\left(\epsilon\right)\,\,\,\epsilon\rightarrow0.
\]
Applied with a sequence $\epsilon^{n}\nearrow0,$ it follows after
dividing by $\epsilon^{n}$ and letting $n\rightarrow\infty$ that
$\partial_{t-}v\left(t,x\right)\leq a$. If $\partial_{t-}v\left(t,x\right)<\partial_{t+}v\left(t,x\right)$
then for $\left(a,p,m\right)\in\mathcal{P}_{\mathcal{O}}^{2,+}v\left(t,x\right)$
we have 
\[
v\left(t+\epsilon,x\right)-v\left(t,x\right)\leq a\epsilon+o\left(\epsilon\right)\,\,\,\epsilon\rightarrow0
\]
which leads after taking $\epsilon^{n}\nearrow0$ resp.~$\epsilon^{n}\searrow0$
to 
\[
\partial_{t+}v\left(t,x\right)\leq a\leq\partial_{t-}v\left(t,x\right)
\]
and hence contradicts the assumption $\partial_{t-}v\left(t,x\right)<\partial_{t+}v\left(t,x\right)$.
The other statements follow similarly.\end{proof}
\begin{acknowledgement}
PG is grateful for partial support
from the European Research Council under the European Union's Seventh
Framework Programme (FP7/2007--2013)/ERC grant agreement nr.~258237.
and nr.~291244. 
 HO is grateful for partial support
from the European Research Council under the European Union's Seventh
Framework Programme (FP7/2007--2013)/ERC grant agreements nr.~258237
and nr.~291244. 
GdR is also affiliated with CMA/FCT/UNL, 2829-516 Caparica, Portugal.
GdR acknowledges partial support by the \emph{Fundação para a Ciência
e a Tecnologia} (Portuguese Foundation for Science and Technology)
through PEst-OE/MAT/UI0297/2011 and PEst-OE/MAT/UI0297/2014 (CMA -
Centro de Matemática e Aplicações).
\begin{acknowledgement}
The authors would like to thank Alexander Cox, Martin Keller--Ressel,
Jan Ob\l{}ój and Johannes Ruf for helpful conversations.
\end{acknowledgement}
\end{acknowledgement}
\bibliographystyle{plain}
\bibliography{/home/hd/Dropbox/projects/BIBTEX/root,/home/hd/Dropbox/projects/BIBTEX/roughpaths}

\end{document}